\documentclass{article}
\usepackage[utf8]{inputenc}
\usepackage{pdfpages}
\usepackage{url}
\usepackage{euler} 
\usepackage[toc,page]{appendix}
\usepackage{amssymb, amsthm,amsmath}
\usepackage[Conny]{fncychap}
\usepackage{graphicx}
\usepackage{adjustbox}
\usepackage{comment}
\usepackage{etex}
\usepackage{mathrsfs}
\usepackage[english]{babel}

\usepackage[scaled=0.85]{beramono}
\usepackage{tikz-cd}
\usepackage{float}
\usepackage{fancyhdr}

\usepackage[left=0cm,right=0cm,top=2cm,bottom=2cm,margin=3cm]{geometry}
\usepackage{amsthm}
\usepackage{bigints}
\usepackage{tikz}
\usepackage[linktoc=all]{hyperref}
\usepackage[capitalise]{cleveref}
\hypersetup{colorlinks=true,linkcolor=blue,citecolor=red}
\usetikzlibrary{matrix,arrows,decorations.pathmorphing}
\usetikzlibrary{patterns,calc,angles,quotes}
\usepackage{tikz-cd}
\usepackage[utf8]{inputenc}
\tikzset{commutative diagrams/.cd}
\newtheorem{theorem}{Theorem}[subsection]

\newtheorem{corollary}[theorem]{Corollary}

\newtheorem{lemma}[theorem]{Lemma}

\newtheorem{proposition}[theorem]{Proposition}

\theoremstyle{definition}
\newtheorem{definition}[theorem]{Definition}

\newtheorem{claim}[theorem]{Claim}

\newtheorem{example}[theorem]{Example}

\newtheorem{remark}[theorem]{Remark}

\newtheorem{notation}[theorem]{Notation}

\newcommand{\E}{\mathcal{E}}
\newcommand{\D}{\mathcal{D}}

\newcommand{\Q}{\mathbf{Q}}
\newcommand{\Sc}{\op{Sch}}

\newcommand{\Z}{\mathcal{Z}}
\newcommand{\T}{\mathcal{T}}
\renewcommand{\P}{\mathbf{P}}

\newcommand{\op}[1]{\operatorname{#1}}

\newcommand{\U}{\mathcal{U}}

\newcommand{\Schf}{\op{Sch}_{fd}}

\newcommand{\F}{\mathcal{F}}

\renewcommand{\O}{\mathcal{O}}

\newcommand{\niI}{\underline n}
\newcommand{\ep}{\epsilon}
\newcommand{\dd}{\delta}
\newcommand{\I}{\mathscr{I}}
\newcommand{\A}{\mathbf{A}}
\newcommand{\Ca}{\mathcal{C}}

\newcommand{\X}{\mathcal{X}}

\newcommand{\bb}{\bullet}

\numberwithin{subsection}{section}
\newcommand{\Sp}{\op{Spec}}

\newcommand{\Set}{\op{Sets}}
\newcommand{\J}{\mathcal{J}}

\newcommand{\M}{\mathcal{M}}

\newcommand{\sset}{\op{Set}_{\Delta}}

\newcommand{\V}{\mathcal{V}}

\renewcommand{\U}{\mathcal{U}}
\newcommand{\bx}{\square}

\newcommand{\Nst}{N^D_{\bb}(\op{Nis-locSt})}
\newcommand{\Nstalg}{N(\op{Algsp})}

\newcommand{\SHe}{\mathcal{SH}_{\op{ext}}}
\newcommand{\Sheshr}{\mathcal{SH}_{\op{ext}!}}
\newcommand{\Sheshrup}{\mathcal{SH}^!_{\op{ext}}}

\newcommand{\Shext}{{\mathcal{SH}}^{\otimes}_{\op{ext}}}
\newcommand{\SAlg}{{\mathcal{SH}}^{\otimes}_{\op{Algsp}}}

\newcommand{\Sho}{\mathcal{SH}^{\otimes}}
\newcommand{\SH}{\mathcal{SH}}
\newcommand{\Y}{\mathcal{Y}}
\fancypagestyle{main}{\fancyhf{}
	\fancyhead[LE,RO]{\scriptsize SH for algebraic stacks.}
	\fancyhead[RE,LO]{\scriptsize \rightmark }
	\fancyfoot[CE,CO]{\thepage}}

\title{Motivic Homotopy Theory of Algebraic Stacks}
\author{Chirantan Chowdhury}
\date{\today}

\begin{document}

\maketitle{}
\begin{abstract}
 	 The aim of this paper is to extend the definition of motivic homotopy theory from schemes to a large class of algebraic stacks and establish a six functor formalism. The class of algebraic stacks that we consider includes many interesting examples: quasi-separated algebraic spaces, local quotient stacks and moduli stacks of vector bundles. We use the language of $\infty$-categories developed by Lurie. Morever, we use the so-called 'enhanced operation map' due to Liu and Zheng to extend the six functor formalism from schemes to our class of algebraic stacks. We also prove that six functors satisfy properties like homotopy invariance, localization and purity.
	
\end{abstract}	
\tableofcontents

\section{Introduction}
	
	The six functor formalism was formulated by Grothendieck to give a framework for the basic operations and duality statements for cohomology theories. In brief, a six functor formalism is a theory of coefficient systems relative to any scheme with a collection of six functors $f^*,f_*,f^!,f_!,\otimes, \op{Hom}$ which satisfy a set of relations. This formalism is usually formulated in the language of triangulated categories.  In \cite{MorVod}, Morel and Voevodsky define the general theory of $\A^1$-homotopy theory of schemes which incorporates homotopy theory in the field of algebraic geometry. To a scheme $S$, they associate a triangulated category $\op{SH}(S)$ which is defined by applying $\A^1$-localization and $\P^1$-stabilization to the category of simplicial Nisnevich sheaves. Voevodsky and Ayoub (\cite{Ay1} and \cite{Ay2}) constructed a six functor formalism of $\A^1$-homotopy theory. In this paper, we extend the definition of $\SH$  to a large class of algebriac stacks and provide a six functor formalism for $\SH$ using the language of $\infty$-categories developed by Lurie (\cite{HTT} and \cite{HA}).\\
	
	In order to motivate the need of language of $\infty$-categories, let us recall the six functor formalism of derived categories of $\ell$-adic sheaves over an algebraic stack. To an algebraic stack $\X$, one can try to define the derived category of the algebraic stack $\X$ as derived category of $\ell$-adic \'etale sheaves over $\X$. Let $\X =B\mathbf{G}_m$. With this definition, we would get that the derived category of $B\mathbf{G}_m$ is the derived category of $\mathbf{G}_m$-equivariant $\ell$-adic \'etale sheaves over a point. As the connected group $\mathbf{G}_m$ cannot act non-trivially on locally constant sheaves, this is equivalent to the category of sheaves over a point. This definition does not give the result as one expects for the classifying stack of bundles as we have  \[ H^*(B\mathbf{G}_m) \cong \mathbf{Q}_l[c] \] where $c$ is in degree $2$ (\cite{totarochowring}).\\
	
	In \cite{lazloolsson1} and \cite{lazloolsson2}, Laszlo and Olsson define derived categories of algebraic stacks and construct the six functor formalism using the lisse-\'etale topos. They use simplicial methods to to construct the derived category that gives the expected answer for the cohomology of $B\mathbf{G}_m$. The fact that the lisse-\'etale topos is not functorial makes the construction of derived pullback technical. The language of $\infty$-categories allows us to circumvent this problem. \\
	
	In \cite{liu2017enhanced}, Liu and Zheng construct a six functor formalism of derived $\infty$-categories of $\ell$-adic sheaves for any algebraic stack. To any scheme $X$, the derived $\infty$-category $\D_{et}(X,\Q_l)$ is the $\infty$-categorical enhancement of the usual derived category. The major advantage of the $\infty$-categorical language is that the derived $\infty$-category satisfies \'etale descent. For any algebraic stack $\X$, the $\infty$-category $\D_{et}(\X,\Q_l)$ constructed by Liu and Zheng is isomorphic to the limit of derived $\infty$-categories over \v{C}ech nerve of any atlas $x: X \to \X$. In other words, we have
	
	\begin{equation}\label{derivedequation}
		\D_{et}(\X,\Q_l) \cong \op{lim}\Big( \begin{tikzcd}
			\D_{et}(X,\Q_l) \arrow[r, shift left =2] \arrow[r, shift right =2] & \D_{et}(X \times_{\X} X,\Q_l) \arrow[l] \arrow[r, shift right =4] \arrow[r] \arrow[r, shift left =4] & \D_{et}(X \times_{\X} X \times_{\X} X,\Q_l) \arrow[l, shift left=2] \arrow[l, shift right =2] & \cdots
		\end{tikzcd}  \Big)
	\end{equation} 
	where the maps in the limit are the derived pullback maps.
	Their construction uses abstract descent theory of the language of $\infty$-categories. This also allows to construct the pullback functor in a canonical way. Morever, they prove that their formalism agrees with the one introduced by Laszlo and Olsson once one passes to homotopy categories of the derived $\infty$-categories. Thus the language of $\infty$-categories seem advantageous to extend $\infty$-sheaves from schemes to algebraic stacks. We shall use a similar technique in our setting of motivic homotopy theory but in this case extra care is needed because motivic invariants usually do not satisfy \'etale descent.\\
	
	To a scheme of finite Krull dimension $S$, the motivic stable homotopy category $\Sho(S)$ is a presentable stable symmetric monoidal $\infty$-category (we refer to \cite{robalo2013noncommutative} for the notations). The functorial assignment makes $\Sho$ into a functor 
	\begin{equation}
		\Sho: N(\Schf)^{op} \to \op{CAlg}(\op{Pr}^L_{stb})
	\end{equation}
	where the target is the $\infty$-category of stable presentable symmetric monoidal $\infty$-categories. As mentioned above, we cannot use equation \cref{derivedequation} as a definition of $\Sho$ for an algebraic stack because $\Sho$ does not satisfy \'etale descent and thus \cref{derivedequation} would depend on the choice of the atlas $X$. We resolve this problem by specifying a class of smooth atlases for which we can prove descent.
	The  resulting class of $(2,1)$-category of algebraic stacks \textbf{$\op{Nis-locSt}$} consists of algebraic stacks which admit an atlas admitting Nisnevich-local sections. This includes all quasi-separated algebraic spaces, quotient stacks $[X/G]$ where $G$ is an affine algebraic group, local quotient stacks, the moduli stack of vector bundles $\op{Bun}_n$, the moduli stack of $G$-bundles $\op{Bun}_G$ and moduli space of stable maps. Using the formulation of enhanced operation map (\cite{liu2017enhanced}), we also manage to extend the six functor formalism from schemes to $\op{Nis-locSt}$. Our main result is as follows (see \cref{thmsixopext} for a complete statement):\\
	\begin{theorem}
		The functor $\Sho(-)$ extends to a functor  \[ \Shext: N^D_{\bb}(\op{Nis-locSt})^{op} \to \op{CAlg}(\op{Pr}^L_{stb}). \] Morever,
		\begin{enumerate}
			\item For any $ \X \in \op{Nis-locSt}$, there exist functors $\otimes, \op{Hom}: \SHe(\X) \times \SHe(\X) \to \SHe(\X)$.
			\item For any morphism  $ f: \X \to \Y$ in $\op{Nis-locSt}$, there is a pair of adjoint functors 
			\[ f^*: \SHe(\Y) \to \SHe(\X) ~,~ f_*: \SHe(\X) \to \SHe(\Y). \]
			\item For a morphism $ f: \X \to \Y$ in $\op{Nis-locSt}$ which is separated of finite type and representable by algebraic spaces, there is a pair of adjoint functors 
			\[ f_!: \SHe(\X) \to \SHe(\Y) ~,~ f^!: \SHe(\Y) \to \SHe(\X). \] 
		\end{enumerate}  
		These functors restrict to the known functors on the category of schemes. Furthermore, the projection formula, base change, localization, homotopy invariance and purity extend to $\op{Nis-locSt}$. 
	\end{theorem}
	
Hoyois defines $\op{SH}$ for quotient stacks of the form $[X/G]$ where $G$ is tame (\cite{Equsixop}). His construction apriori depends on the presentation of the stack. Our construction allows us to drop the tameness assumption for quotient stacks and provides a version of $\SH$ that does not depend on the choice of a presentation. \\
Khan and Ravi define limit-extended cohomology theories for derived stacks (\cite[Section 12]{khan2021generalized}). We believe that our construction agrees with their limit extended construction (see \cref{limitextendedcomparison} for more details).\\
	
	We now give a brief outline of the chapters in the thesis.
	
	\begin{enumerate}
		
		\item In Section 2, we enhance the motivic stable homotopy functor from schemes to algebraic stacks. This follows from \cref{thmmain} which is set in an abstract setup of categories of stacks admitting $\T$-local sections (\cref{stacklocsecdef}). The key example of categories of stacks admitting $\T$-local sections is the $(2,1)$-category $\op{Nis-locSt}$. It is important to note that \cref{thmmain} is a special case of \cite[Proposition 4.1.1]{liu2017enhanced}. We give a new proof of the theorem which is partly inspired from the proof of Liu and Zheng. The extension theorem also allows to construct the four functors $f_*,f^*, \op{Hom}$ and $\otimes$. 
		\item In Section 3, we construct the exceptional functors $f_!$, $f^!$ and also prove base change and projection formulas. This is proved by the so called enhanced operation map introduced by Liu and Zheng in \cite{liu2017enhanced}. In short, we extend the enhanced operation from schemes to $\op{Nis-locSt}$ in the setting of categories of stacks admitting $\T$-local sections (\cref{thmmain2}). The extension is a special case of the DESCENT program (\cite[Chapter 4]{liu2017enhanced}) and we give a new proof of the extension of the enhanced operation from schemes $\op{Nis-locSt}$ which gives us the exceptional functors, projection formula and base change.
		\item In Section 4, we prove the other relations of the six functors namely smooth and proper base change, localization, homotopy invariance and purity. We then collect all the results in a single theorem (\cref{thmsixopext}).
		
	\end{enumerate} 
		\subsection*{Acknowledgments}
		The work presented in this paper was carried during author's PhD thesis at the University of Duisburg-Essen supported by SPP 1786 "Homotopy Theory and Algebraic Geometry". The author would like to thank his supervisor Jochen Heinloth for his invaluable advice, continuous support and patience during the PhD study. The author would also like to thank Marc Levine and Markus Spitzweck for valuable comments regarding the paper. The author expresses his gratitude to all members of ESAGA group in Essen. The paper was completed while the author was a Post Doc at the University of Duisburg-Essen, supported by the ERC Grant "Quadratic Refinements in Algebraic Geometry".

\section{Enhancement of sheaves along coverings with local sections}
In this section, we prove \cref{thmmain} which enables us to extend sheaves from schemes to algebraic stacks. This key point is that for cohomology theories like $\SH$, descent may not be true for general smooth morphisms, which are used as atlases for algebraic stacks. But if the atlas admits local sections in a topology that is coarser than the smooth topology, then it is plausible to extend sheaves from schemes to a large class of Artin stacks.\\
At first, we introduce the notion of $\T$- local sections associated to a site $(\Ca,\T)$.  Then we define a $(2,1)$-category called "category of stacks admitting $\T$-local sections". This allows us to introduce the $(2,1)$-category $\op{Nis-locSt}$ for which we extend cohomology theories like $\SH$. The abstract formalism of category of stacks admitting $\T$-local sections help us to prove the \cref{thmmain} and define the stable homotopy theory on $\op{Nis-locSt}$.

\subsection{Morphisms admitting $\T$-local sections}
	
	Let $\Ca$ be a category admitting products and small coproducts equipped with a Grothendieck topology $\T$. For any $Y \in \Ca$, let $\op{Cov}(Y)$ be the collection of coverings of $Y$.

	\begin{definition}
		A morphism $f: X \to Y$ in $\Ca$ admits \textit{$\T$-local sections}  if there exists a family $\{ p_i: Y_i \to Y\}_{i \in I} \in \op{Cov}(Y)$  and morphisms $s_i: Y_i \to X$ such that the  diagram 
		
		\begin{equation}
			\begin{tikzcd}
				{} & X \arrow[d,"f"] \\
				\coprod_i Y_i \arrow[ur,"\coprod s_i"] \arrow[r,"\coprod p_i"] & Y
			\end{tikzcd}
		\end{equation}
		commutes.

	\end{definition}

	\begin{example}\label{smthetale}
		In the category of schemes equipped with  the \'etale topology, any smooth surjective morphism admits \'etale local sections. This follows from \cite[\href{https://stacks.math.columbia.edu/tag/055U}{Tag 055U}]{stacks-project}.
	\end{example}
	
	\begin{notation}
	    We shall abbreviate the coverings $\coprod_i Y_i$ by $\tilde{Y}$ in the proofs. 
	\end{notation}
	
	\begin{lemma}\label{tlocalpullbackinsch}
		Morphisms admitting $\T$-local sections are stable under pullbacks and compositions. 
	\end{lemma}
	\begin{proof}
		\textbf{Stable under pullbacks}: Let $f: X \to Y$ be a morphism  admitting $\T$-local section and let $ g: Y' \to Y$ be a morphism in $\Ca$. We denote the pullback of $f$ along $g$ by $f': X':= X \times_Y Y' \to Y'$. We need to show that $f'$ admits a $\T$-local section. \\
		As $f$ admits $\T$-local sections, there exists a covering $p: \tilde{Y} \to Y$ and a morphism $s: \tilde{Y} \to X$ such that $ p= f \circ s$.\\
		Then  $p':\tilde{Y'}:= Y' \times_Y \tilde{Y} \to Y'$ is a covering and $s':=\op{id} \times s: \tilde{Y'} \to X'$ satisfies $p' =f' \circ s'$. Thus $f'$ admits $\T$-local sections.\\
		
		\textbf{Stable under compositions}: Let $ f: X \to Y$ and $ g: Y \to Z$ be morphisms admitting $\T$-local sections. We need to show that $ g \circ f$ admits $\T$-local sections. As $g$ admits $\T$-local sections, we have a covering $q:\tilde{Z} \to Z$ and a section $s: \tilde{Z} \to Y$.\\
		By $(1)$, the morphism $f':= (Z'= \tilde{Z}\times_Y X) \to \tilde{Z}$ admits $\T$-local sections. So there exists a covering $p: \tilde{Z''} \to \tilde{Z}$ admitting a section $s'': \tilde{Z''} \to \tilde{Z'}$.\\
		
		These morphisms give rise to a commutative diagram 
		\begin{equation}
			\begin{tikzcd}
				{} & {} & X \arrow[d,"f"] \\
				{} & \tilde{Z'}  \arrow[d,"f'"] \arrow[ur,"s'"] & Y \arrow[d,"g"]\\
				\tilde{Z''} \arrow[r,"p"] \arrow[ur,"s''"] & \tilde{Z} \arrow[ur,"s"] \arrow[r,"q"] & Z,  	
			\end{tikzcd}
		\end{equation} 
		i.e. $s' \circ s''$ is a $\T$-local section of $g \circ f$.
		
	\end{proof}
	\begin{corollary}
		The category $\Ca$ with the set of coverings as
		\[\op{Cov}_{\op{\T-{loc}}}(X):= \{ \{x_i: X_i \to X\} | \coprod_i x_i ~\text{admits}~\T-\text{local sections} \} \] defines a site.
	\end{corollary}
	\begin{proof}
		By definition, identity morphisms admit $\T$-local sections. As morphisms admitting $\T$-local sections are stable under pullbacks and compositions (\cref{tlocalpullbackinsch}), we see that $\Ca$ with the coverings $\op{\T-loc}$ forms a site.
	\end{proof}
 It turns out that sheaves satisfy descent along morphisms that admit $\T$-local sections. In order to prove such a statement, let us recall the notion of $F$-descent.
		\begin{definition} \label{fdescent} \cite[Defintion 3.1.1]{liu2017enhanced}
		Let $\Ca$ be an $\infty$-category which admits pullbacks. $F: \Ca^{op} \to \D$ be a functor of $\infty$-categories and $f: X_0^{+} \to X_1^+$ be a morphism in $\Ca$. Then $f$ satisfies \textit{$F$-descent} if 
		\[ F \circ (X^{+}_{\bb})^{op}: N(\Delta_+) \to \D  \] is a limit diagram where $X^{+}_{\bb}$ is the \v{C}ech nerve of $f$.
	\end{definition}
	The following lemma shall be used to prove descent statements.
		\begin{lemma}\label{equivprodlemma}\cite[Lemma 4.2.2]{liu2017enhanced}
		Let $\D$ be an $\infty$-category which admits products. Let 
		\begin{center}
			\begin{tikzcd}
				D_3 \arrow[r,"f_{32}"] \arrow[d,"f_{31}"] & D_2 \arrow[d,"f_{20}"] \\D_1 \arrow[r,"f_{10}"] & D_0 
			\end{tikzcd}
		\end{center}
		be a commutative diagram in $\Ca$.  Let $F: \D^{op} \to \D'$ be a functor of infinity categories. Assume the following:
		\begin{enumerate}
			\item $f_{32}$ and $f_{31}$ satisfies $F$-descent. 
			\item $f_{20}$ satisfies $F$-descent.
		\end{enumerate}
		Then $f_{10}$ satisfies $F$-descent. Also, we have:
		
		\begin{equation}\label{equivprodlim}\op{lim}_{n \in \Delta^{op}} F(D_{1n}) \xrightarrow{\cong} \op{lim}_{n \in \Delta^{op}}F(D_{3n})  \xleftarrow{\cong} \op{lim}_{m \in \Delta^{op}} F(D_{2m}) 
		\end{equation}
		where $D_{1n}:= D_1^{\times n}$ over $D_0$, $D_{2m}:= D_2^{\times m} $ over $D_0$ and $D_{3n}:= D_{1n} \times_{D_0} D_{2n}$.
	\end{lemma}
	
	\begin{proposition} \label{covlocsecdesc}
		Let $F: N(\Ca)^{op} \to \D$ an $\infty$-sheaf. Then $F$ satisfies descent along morphisms that admit $\T$-local sections.
	\end{proposition}
	\begin{proof}
		Let $f: X \to Y$ be a morphism in $\Ca$ which admits $\T$-local sections. Thus there exists a covering $\varphi: \tilde{Y} '\to Y$  and a section $ s': \tilde{Y} \to \tilde{X}:= \tilde{Y} \times_{Y} X$  of $f': \tilde{X} \to \tilde{Y}$ such that the diagram
		
		\begin{center}
			\begin{tikzcd}
				\tilde{X} \arrow[r,"\varphi'"] \arrow[d,"f'"] & X \arrow[d,"f"]\\
				\tilde{Y}\arrow[r,"\varphi"] \arrow[u,bend left =80, "s'"] & Y \\
			\end{tikzcd}
		\end{center}
		commutes in $\Ca$.\\
		As $f'$ admits a section $s'$, the \v{C}ech nerve of $f$ is a split-simplicial object (\cite[Definition 4.7.2.2]{HA}). As split-simplicial objects are colimit diagrams (\cite[Lemma 6.1.3.16]{HTT}), $f'$ satisfies $F$-descent . \\
		As $F$ is  an $\infty$-sheaf, the horizontal arrows have $F$-descent. As $f',\varphi$ and $\varphi'$ satisfy $F$-descent, by \cref{equivprodlemma} we get that $f$ satisfies $F$-descent.
	\end{proof}
	\begin{remark}
		When $\Ca = \op{Sch}$ and $\T =\text{\'et}$, the above proposition gives us that the category of \'etale sheaves and category of smooth sheaves are equivalent. 
	\end{remark}
	\subsection{Categories of stacks admitting $\T$-local sections.}
	
	We want to extend sheaves on schemes to the $(2,1)$-category of algebraic stacks. This extension is a two step process. We first extend from schemes to algebraic spaces and then from algebraic spaces to algebraic stacks. In order to formalize the statements in a coherent manner, we define an abstract $(2,1)$-category $\op{St}_{\Ca}$ which incorporates the properties of algebraic spaces and algebraic stacks.\\
	
	\begin{definition}\label{stacklocsecdef}
		Let $\Ca$ be a category equipped with a Grothendieck topology $\T$. A \textit{category of stacks admitting $\T$-local sections} is a $(2,1)$-category $\op{St}_{\Ca}$ together with a fully faithful inclusion $i_{\Ca}: \Ca \hookrightarrow \op{St}_{\Ca}$ satisfying the following properties:
		\begin{enumerate}
			\item $\op{St}_{\Ca}$ admits fiber products and small coproducts.
			\item Given any object $ \X \in \op{St}_{\Ca}$, there exists a morphism $x: X \to \X$ with $X \in \Ca$ such that for any morphism $ y: X' \to \X$ in $\Ca$, the fiber product $ x': X' \times_{\X} X \to X'$ admits $\T$-local sections. We say that $x$ as an \textit{atlas admitting $\T$-local sections}.
			\item The diagonal $\X \to \X \times \X$ is representable in $\Ca$. 
		\end{enumerate}
		
	\end{definition}
	\begin{remark}
		Here representablity is understood as for algebraic stacks i.e. a morphism $ f: \X \to \Y$ in $\op{St}_{\Ca}$ if for all $ Y \to \Y$ with $Y \in \Ca$, the fiber product $\X \times_{\Y} Y$ is in $\Ca$. With this definition, the diagonal map being representable is equivalent of saying that any morphism $x: X \to \X$ where $X \in \Ca$ is representable. \\ As in the case of algebraic stacks, representable morphisms are stable under pullbacks.
	\end{remark}	
	
	\begin{definition}\label{repandtlocdef}
		Let $\op{St}_{\Ca}$ be a category of stacks admitting $\T$-local sections.\\
		A morphism $ f: \X \to \Y$ is said to admit \textit{$\T$-local sections} if there exists a an atlas $ y: Y \to \Y$ and a morphism $ s: Y \to \X$ such that $f \circ s = y$.	
		
	\end{definition}
	
	The following gives a simpler definition for morphisms admitting $\T$-local sections in the setting of representable morphisms.
	\begin{lemma}\label{tlocrep}
		A representable morphism $f: \X \to \Y$ in $\op{St}_{\Ca}$ admits $\T$-local sections iff for any morphism $ v: V \to \Y$ with $V \in \Ca$, the base change morphism $ f': V \times_{\Y} \X \to V$ is a morphism admitting $\T$-local sections in $\Ca$.
	\end{lemma}
	\begin{proof} Let $f$ be a representable morphism admitting $\T$-local sections. Let $ v: V \to \Y$ be a morphism where $V \in \Ca$. We want to show that the base change morphism $f': X':= V \times_{\Y} \X \to V$ is a morphism admitting $\T$-local sections. By definition, there exists an atlas admitting $\T$-local sections $ v': V' \to \Y$ and a section $s: V' \to \X$. Then the base change morphism $v'': V'':= V' \times_{\Y} V \to V$ admits $\T$-local sections. \\
		On the other hand, the section $s$ induces a map $s': V'' \to X'$. As $v''$ admits $\T$-local sections, there exists a covering $ \tilde{v}: \tilde{V} \to V$ and a section $s'': \tilde{V} \to V''$. \\
		Thus there exists a covering $\tilde{v}$ and a section $s' \circ s''$ implying that $f'$ admits $\T$-local sections. \\
		
		For the other direction, let $ v: V \to \Y$ be an atlas admitting $\T$-local sections. The assumptions says that the base change morphism $f': X':= \X \times_{\Y} V$ is a morphism admitting $\T$-local sections. Thus there exists a covering $v': V' \to V$ and a section $s': V' \to X'$. Then the compositions $v \circ v'$ and $x' \circ s'$ imply that $f$ admits $\T$-local sections (here $x'$ is the pullback of $v$ along $f$).  
	\end{proof}
	\begin{lemma}
		The pullback of an atlas along any morphism in $\op{St}_{\Ca}$ is a morphism admitting $\T$-local sections.
		
	\end{lemma}
	\begin{proof}
		Let $y: Y \to \Y$ be an atlas admitting $\T$-local sections and let $f: \X \to \Y$ be a morphism in $\op{St}_{\Ca}$. We want to show that $y': \X':=Y \times_{\Y} \X \to \X$ admits $\T$-local sections. As $v'$ is representable, by \cref{tlocrep} it suffices to show that for any morphism $ x: X \to \X$ where $X \in \Ca$, the fiber product $y'': X \times_{\X} \X' \to X$ admits $\T$-local sections. This follows from the fact that $y''$ is pullback of $y$ along $ y \circ x$ and $y$ is an atlas. 
	\end{proof}
	\begin{lemma} \label{covpullbandcomp}
		Morphisms in $\op{St}_{\Ca}$ admitting $\T$-local sections are stable under pullbacks and compositions.
	\end{lemma}
	\begin{proof}
		\begin{enumerate}
			\textbf{Stable under pullbacks}: Consider a pullback square 
			\begin{equation}
				\begin{tikzcd}
					\X' \arrow[r] \arrow[d,"f'"] & \X \arrow[d,"f"] \\
					\Y' \arrow[r] & \Y
				\end{tikzcd}
			\end{equation}
			where $f$ admits $\T$-local sections. Thus, there exists an atlas $y: Y \to \Y$ and a morphism $ g: Y \to \X$ such that $f \circ g = y$. Then the base change morphism $\tilde{y}: \Y'':= Y \times_{\Y} \Y' \to \Y'$ admits $\T$-local sections. As $\Y'$ admits a morphism to $\X$ (via $g$), we have a unique morphism $ g'': \Y'' \to \X'$ such that $f' \circ g'' = \tilde{y}$. We denote $y': Y' \to \Y'' \xrightarrow{\tilde{y}} \Y''$ to be the composition where $Y'$ is an atlas of $\Y''$. Note that $y'$ admits $\T$-local sections. Denote $g'$ to be composition $ g': Y' \to \Y'' \xrightarrow{g''} \X''$. Then we have that $f' \circ g' = y'$ thus proving the fact that $f'$ admits $\T$-local sections. \\
			
			\textbf{Stable under compositions}: Let $ f: \X \to \Y$ and $g: \Y \to \Z$ be morphisms admitting $\T$-local sections. Thus there exists atlases $ z: Z \to \Z$, $y: Y \to \Y$ and morphisms $ p: Z \to Y$, $ q: Y \to \X$ such that $ g \circ p = z$ and $ f \circ q = y$. Let  $p':  Z':= Y \times_{\Y} Z \to Y$ and $y': Z' \to Z$ be the base change of $p$ and $y$ respectively. Thus $y'$ admits $\T$-local sections. We denote the compositions by $z':= z \circ y': Z' \to \Z$ and $q':= q \circ p': Z' \to \X$.  Thus we get that $ q' \circ(g \circ f) =z'$ implying  that $ g\circ f$ admits $\T$-local sections (as $z'$ is an atlas of $\Z$). 
		\end{enumerate}
	\end{proof}
	The following lemma gives us another definition of morphisms admitting $\T$-local sections which shall help us to prove the sheaf condition in \cref{thmmain}. 
	\begin{lemma}\label{tlocalsectionanotherformulation}
		A morphism $f: \X \to \Y$ admits $\T$-local sections iff there exists a commutative diagram 
		\begin{equation}
			\begin{tikzcd}
				X \arrow[r,"x"] \arrow[d,"f'"] & \X \arrow[d,"f"] \\
				Y \arrow[r,"y"] & \Y
			\end{tikzcd}
		\end{equation}
		where $f'$ admits $\T$-local sections in $\Ca$ and $x$, $y$ are atlases admitting $\T$-local sections.
		
	\end{lemma}
	\begin{proof}
		Let $f$ be a morphism admitting $\T$-local sections. Let $y: Y \to \Y$ be an atlas admitting $\T$-local sections. Then the base change morphisms $f'': \X':= \X \times_{\Y} Y \to Y$ and $x': \X' \to \X$ admit $\T$-local sections.  Let $x'': X \to \X'$ be an atlas of $\X'$. Then the compositions $ f':= f'' \circ x''$ and $x:= x' \circ x''$ admit $\T$-local sections giving us the commutative diagram that was needed.\\
		For the other direction, consider a commutative square 
		\begin{equation}
			\begin{tikzcd}
				X \arrow[r,"x"] \arrow[d,"f'"] & \X \arrow[d,"f"] \\
				Y \arrow[r,"y"] & \Y
			\end{tikzcd}
		\end{equation}	
		where $f'$ admits $\T$-local sections and $x$, $y$ are atlases. By definition there exists a covering $y': Y' \to Y$ in the topology $\T$ and a morphism $s': Y' \to X$ such that $f \circ s = y'$.  Defining $s = x \circ s'$ and $y'': Y' \to Y \to \Y$ we get that $ f \circ s = y''$ where $y''$ is an atlas of $\Y$. Hence $f$ admits $\T$-local sections.
		
	\end{proof}
	
	For any object $\X \in \op{St}_{\Ca}$, define $\op{Cov}(\X)$ as the set of families of the form $\{x_i: \X_i \to \X \}_{i \in I}$ such that $ x:= \coprod_i x_i: \coprod_i \X_i \to \X $  admits $\T$-local sections. \\
	
	\begin{lemma}\label{stcsitelemma}
		The family of coverings admitting $\T$-local sections $\op{Cov}(\X)$ for every object $\X \in \op{St}_{\Ca}$ defines a Grothendieck topology on $\op{St}_{\Ca}$. We will write $(\op{St}_{\Ca},\op{\T-\op{loc}})$ for the corresponding site.
	\end{lemma}
	\begin{proof}
		The identity morphisms in $\op{St}_{\Ca}$ admit $\T$-local sections. By \cref{covpullbandcomp}, $\T$-local sections are stable under pullbacks and compositions. Thus the $(2,1)$-category $\op{St}_{\Ca}$ defines a site.
		
	\end{proof}
	
	\subsection{The $(2,1)$-category $\op{Nis-locSt}$.}\label{Nislocdefinition}
	
	In this section, we introduce the class of stacks for which we can extend cohomology theories that satisfy descent with respect to the Nisnevich topology. This class of stacks will be called $\op{Nis-locSt}$ and we will explain that many interesting Artin stacks are contained in this class. In particular, we show that this contains all local quotient stacks and quasi-separated algebraic spaces.\\
	
	At first, we consider $\Ca =\op{Sch}$, the category of schemes equipped with the Nisnevich topology. Then any quasi-separated algebraic space has an atlas admitting Nisnevich-local sections (\cite[Chapter 2, Theorem 6.3]{knu}). So we can consider $\op{St}_{\Ca} = \Nstalg$ to be the category of quasi-separated algebraic spaces.
	\begin{remark}\label{algspcover}
		By the above discussion, given any quasi-separated algebraic space $\X$, there exists a Nisnevich covering $x: X \to \X$ where $X$ is a scheme.
	\end{remark}
	
	\begin{notation}
		The category $\op{Nis-locSt}$ is the category of algebraic stacks for which there exists a smooth atlas admitting Nisnevich-local sections. In the terminology introduced in \cref{stacklocsecdef}, this is the category of stacks admitting Nisnevich-local sections for the category $\Nstalg$ of quasiseparated algebraic spaces.
	\end{notation}

	Before listing some examples of algebraic stacks in $\op{Nis-locSt}$, let us verify some properties of the category $\op{Nis-locSt}$. 
	
	\begin{lemma}\label{nislocstproperty1}
		The $(2,1)$-category $\op{Nis-locSt}$ admits fiber products.
	\end{lemma}
	\begin{proof}
		Let 
		\begin{equation}
			\begin{tikzcd}
				{} & \X_2 \arrow[d] \\
				\X_1 \arrow[r] & \X_0 
			\end{tikzcd}
		\end{equation}
		be a diagram where $\X_0,\X_1$ and $\X_2$ are objects in $\op{Nis-locSt}$. We want to show that $\X_3:=\X_1 \times_{\X_0} \X_2$ is algebraic stack which has an atlas admitting Nisnevich-local sections. \\
		Let $ x_0: X_0 \to \X_0$ be an atlas admitting Nisnevich-local sections. At first, we prove the following claim
		\begin{claim}\label{covnonemptyclaim}
			Let $ f: \X \to \Y$ be a morphism in $\op{Nis-locSt}$, then for every atlas $y:Y \to Y$, there exists of a 2-commutative square 
			\begin{equation}
				\begin{tikzcd}
					X \arrow[r,"x"] \arrow[d,"f'"] & \X \arrow[d,"f"] \\
					Y \arrow[r,"y"] & \Y
				\end{tikzcd}
			\end{equation}
			where $f'$ is a morphism of schemes and $x$ is  an atlas admitting Nisnevich-local sections.
		\end{claim}
		\begin{proof}[Proof of claim]
			Let $x: X_0 \to \X$ be an atlas of  $\X$. Then the base change $x': X':= \X \times_{\Y} X \to \X$ is a morphism admitting $\T$-local sections. Therefore $x:X:= X_0  \times_{\X} X' \to X' \xrightarrow{x'} \X$ is an atlas of $\X$ fitting in the diagram 	
			\begin{equation}
				\begin{tikzcd}
					X \arrow[r,"\tilde{y}"] \arrow[d,"f'"] & \X \arrow[d,"f"] \\
					Y \arrow[r,"y"] & \Y
				\end{tikzcd}
			\end{equation}
			where $f$ lies in $\Ca$ and $y,x$ are morphisms admitting $\T$-local sections.\\
		\end{proof}
		
		Applying the claim, we get that there exist atlases $x_1: X_1 \to \X_1$ and $ x_2: X_2 \to \X_2$ such that the following diagrams 
		\begin{equation}
			\begin{tikzcd}
				X_1 \arrow[r] \arrow[d,"x_1"] & X_0 \arrow[d,"x_0"] \\
				\X_1 \arrow[r] & \X_0,
			\end{tikzcd}
			\begin{tikzcd}
				X_2 \arrow[r] \arrow[d,"x_2"] & X_0 \arrow[d,"x_0"] \\
				\X_2 \arrow[r] & \X_0,
			\end{tikzcd}
		\end{equation}
		commute. This induces a natural map $x_3: X_3:= X_1 \times_{X_0} X_2 \to \X_3$. We claim that $x_3$ is an atlas admitting Nisnevich-local sections. Let $v_3: V \to \X_3$ be a morphism where $V$ is a scheme. Then it induces maps $v_1: V \to \X_1$, $v_0: V  \to \X_0$ and $v_2: V \to \X_2$. Thus the base change morphisms $x'_1: X'_1:=V \times_{\X_1} X_1 \to V$, $x'_2:X'_2:=V \times_{\X_2} X_2 \to V$ and $x'_0:X'_0:=V \times_{\X_0} X_0 \to V$ admit Nisnevich-local sections. As the fiber product $V \times_{\X_3} X_3 \cong X'_1 \times_{X'_0} X'_2$, the morphism $x'_3: V \times_{\X_3} X_3 \to V$ admits Nisnevich-local sections. Thus $x_3$ is an atlas admitting Nisnevich-local sections. 
	\end{proof}
	\begin{lemma}\label{nislocproperty2}
		Let $ f: \X \to \Y$ be a  morphism of algebraic stacks representable by algebraic spaces such that $\Y \in \op{Nis-locSt}$, then $\X \in \op{Nis-locSt}$.
	\end{lemma}
	\begin{proof}
		Let $ y: Y \to \Y$ be an atlas admitting Nisnevich local sections. By \cref{algspcover}, it suffices to show that the base change morphism $x: X:=Y \times_{\Y} \X \to \X$ admits Nisnevich-local sections where $X$ is an algebraic space. Let $x': X' \to \X$ be a morphism where $X$ is a scheme. As $X'':= X' \times_{\X} X \cong X' \times_{\Y} Y$, the base change morphism $y': X'' \to Y$ is a morphism of scheme admitting Nisnevich-local sections. 
	\end{proof}
	Before proving the next corollary, let us recall that an algebraic stack $\X$ is a local quotient stack if it admits a open covering by quotient stacks of the form $[X/G]$ where $G$ is an affine algebraic group ( \cite[A.2.2]{Freed_2011}).  
	\begin{corollary}\label{quotstackinnst}
		All local quotient stacks are contained in $\op{Nis-locSt}$.
	\end{corollary}
	\begin{proof}
		As Zariski open coverings admit Zariski-local sections, it suffices to prove that quotient stacks lie in $\op{Nis-locSt}$.	
		At first, we see that $\op{BGL_n}$ lies in $\op{Nis-locSt}$. This is because the atlas $\op{pt} \to \op{BGL_n}$ is a $\op{GL}_n$-torsor. As $\op{GL}_n$ is special, $\op{GL}_n$-torsors are Zariski-locally trivial. Thus $\op{pt} \to \op{BGL_n}$ is an atlas admitting Nisnevich-local sections.\\
		If $G$ is an affine algebraic group, then the inclusion $i: G \hookrightarrow \op{GL}_n$ induces a representable morphism $i: \op{BG} \to \op{BGL_n}$. As $\op{BGL_n} \in \op{Nis-locSt}$, by \cref{nislocproperty2} we get that $\op{BG} \in \op{Nis-locSt}$.\\
		The map $[X/G] \to \op{BG}$ is representable. Applying \cref{nislocproperty2}, we get that $[X/G] \in \op{Nis-locSt}$. \\
		
	\end{proof}
	\begin{remark}
		Note that unless $G$ is a special group, the standard atlas $X \to [X/G]$ of a quotient stack may not admit Nisnevich-local sections. In the proof above, this atlas is replaced by a scheme $X''$ which is a Nisnevich cover of the algebraic space $X \times^G \op{GL}_n$. Lets us explain this in detail.\\
		We can write $[X/G]$ as $ [X \times^G GL_n/ GL_n]$. The object $X \times^G GL_n$ exists as an algebraic space. Thus the morphism $x':X \times^G \op{GL_n} \to [X/G]$ is a $\op{GL}_n$-torsor and hence admits Nisnevich-local sections. By \cref{algspcover}, we get that there exists a Nisnevich cover $x'':X'' \to X \times^G \op{GL}_n$ where $X''$ is a scheme. Hence, we get that the morphism $x' \circ x'': X'' \to [X/G]$ is an atlas admitting Nisnevich-local sections. \\
		
	\end{remark}
	Recall that Totaro and Gross explained that the property of being a quotient stack is closely related to the resolution property (\cite{totaroresolution} and\cite{grosstensorgenerator}).
	\begin{corollary}
		Let $\X$ be a quasi-compact and quasi-separated algebraic stack which has affine stabilizers at closed points and satisfies the resolution property. Then $\X \in \op{Nis-locSt}$.
	\end{corollary}
	
	\begin{proof}
		Under the assumptions, we get that $\X \cong [U/\op{GL}_n]$ where $U$ is a quasi-affine scheme (\cite[Theorem 1.1]{totaroresolution} and \cite[Theorem A]{grosstensorgenerator}). Thus \cref{quotstackinnst} implies that $\X \in \op{Nis-locSt}$.
	\end{proof}
	
	We now explain how local constructions like blow ups and deformation to normal cone (\cite[Chapter 14]{Laumon2000}) also lie in $\op{Nis-locSt}$. Before stating the corollary, let us briefly recall the notions. Let $\X$ be an algebraic stacks and $ z:\Z \hookrightarrow \X$ be a closed substack and $\I \subset \O_{\X}$ be the ideal sheaf of $\Z$.  For any scheme $T$, we shall denote $T'$ to be the fiber product of $T$ along $z$.
	\begin{notation}\label{blowupdeforstacknotation}

		\begin{enumerate}
			\item The blowup of $\X$ along $\Z$ is the algebraic stack $\op{Bl}_{\Z}(\X):= \op{Proj}(\oplus_{n \geq 0} \I_{\Z}^n)$  which admits a morphism representable by schemes $\op{pr}_{\op{bl}}: \op{Bl}_{\Z}(\X) \to \X$ such that for any morphism $ T \to \X$, the fiber product $\op{Bl}_{\Z}(\X) \times_{\X} T$ is isomorphic to $\op{Bl}_{T'}T$. 
			\item The normal cone $N_{\Z}(\X):=\Sp(\oplus_{n \geq 0} \I^n/\I^{n+1})$ is the algebraic stack which admits a morphism $\op{pr}_{\op{n}}: N_{\Z}(\X) \to \Z$ representable by schemes.
			\item The deformation to the normal cone $D_{\Z}(\X)$ is the analog of deformation space in the setting of schemes. Let us recall the definition in the setting of schemes (\cite[Chapter 6]{Fulton1984}). Let $X$ be a scheme and $Z$ be a closed subscheme of $X$. We have the deformation space
			\[ D_ZX:=  (\widetilde{D_ZX}:= \op{Bl}_{Z \times \{0\}}(X \times \mathbf{A}^1))  -\op{Bl}_ZX\to X \times \mathbf{A}^1 . \]
		     The fiber over $ t= 0$ on $\widetilde{D_ZX}$ ( $t$ being the coordinate of $\A^1$) is the union of two schemes: $\op{Bl}_Z X$ and the projective completion of the normal bundle $\mathbf{P}(N_XY\oplus 1)$ glued along the Cartier divisor $E_X Y$. In this way, the blow up $\op{Bl}_{Z}X$ is realized as a subscheme of $\widetilde{D_ZX}$. \\
		     The algebraic stack $D_{\Z}(\X)$ is defined in the similar way as we do it for schemes (see \cite[Section 6.1]{CycKr} for details) . The algebraic stack $D_{\Z}(\X)$ admits a schematic representable morphism $\op{pr}_d: D_{\Z}(\X) \to \X \times \mathbf{A}^1$ such that the fibers of $\op{pr}_d$ over $\X \times \{0\}$ and $\X \times \{1\}$  are the normal cone $N_{\Z}(\X)$ and the algebraic stack $\X$ respectively. 
			
		\end{enumerate}
	\end{notation}
	\begin{corollary}\label{blownordefnisloc}
		Let $\X \in \op{Nis-locSt}$ and let $\Z$ be a closed substack of $\X$. Then $\op{Bl}_{\Z}(\X),N_{\Z}(\X)$ and $D_{\Z}(\X)$ belong to $\op{Nis-locSt}$.
	\end{corollary}
	\begin{proof}
		As the morphisms $\op{pr}_{\op{bl}}: \op{Bl}_{\Z}(\X) \to \X,\op{pr}_n: N_{\Z}(\X) \to \Z,\op{pr}_d: D_{\Z}(\X) \to (\X \times \mathbf{A}^1)$ are representable, \cref{nislocproperty2} gives us that these algebraic stacks also lie in $\op{Nis-locSt}$. 
	\end{proof}
	
	\begin{corollary}
		\begin{enumerate}
			\item For any projective variety $X$, the stack of vector bundles $\op{Bun}_n$ and the stack of $G$-bundles $\op{Bun}_G$ for an affine algebraic group $G$ are in $\op{Nis-locSt}$. The same result holds for stacks of Higgs bundles $\op{Higgs}_G$.
			
			\item The moduli spaces of stable maps are in $\op{Nis-locSt}$. 
		\end{enumerate}
	\end{corollary}
	\begin{proof}
		\begin{enumerate}
			\item The stack of vector bundles $\op{Bun}_n$ can be written as a union of $\op{Bun}^{\le m}_n$ where $\op{Bun}^m_n$ is the open substack of vector bundles of bounded maximal slope $m$. The stack $\op{Bun}^m_n$ is a locally closed substack of a quotient stack by Quot scheme construction (\cite[Theorem 3.3.7 and Section 4.3]{huybrechts_lehn_2010}). Thus $\op{Bun}_n$ is a local quotient stack and hence by \cref{quotstackinnst} lies in $\op{Nis-locSt}$.\\
			As the morphism $\op{Bun}_G \to \op{Bun}_n$ is representable, we get that $\op{Bun}_G$ is in $\op{Nis-locSt}$ (\cref{nislocproperty2}).
			The same argument holds for Higgs bundles as $\op{Higgs}_G \to \op{Bun}_G$ is representable. 
			\item  The moduli space of stable maps is isomorphic to a quotient stack of the form $[J/\op{PGL}_n]$ where $J$ is a quasi-projective variety(\cite[Section 2.4]{fultonpandharipande}). Hence by \cref{quotstackinnst}, the moduli space of stable maps lies in $\op{Nis-locSt}$.
		\end{enumerate}
		
	\end{proof}

	\subsection{Extension of sheaves from schemes to algebraic stacks.}
	In this section, we state and prove the theorem which helps us to extend $\infty$-sheaves from schemes to algebraic stacks. As we will use \v{C}ech nerves to verify the sheaf condition, we will from now on assume that the categories $\Ca$ and $\op{St}_{\Ca}$ satisfies the conditions in \cite[Proposition A.3.3.1]{SAG} i.e. coproducts are disjoint and finite coproducts are universal. These conditions are satisfied in all of the examples in the previous section. \cref{thmmain} is a special case of \cite[Proposition 4.1.1]{liu2017enhanced}. As this result is crucial for our construction of $\Shext(-)$ and the special case allows for a shorter proof, we give a self-contained proof of the theorem. Before formulating the result, let us recall that an $\infty$-category $\D$ admits geometric realizations if any simplicial object of $\D$ admits a colimit in $\D$.

	\begin{theorem}\label{thmmain}
		Let $(C,\T)$ be a site and $\op{St}_{\Ca}$ a category of stacks admitting $\T$-local sections. Let $F: N(\Ca^{op}) \to \D$ be an $\infty$-sheaf  where $\D^{op}$ is an $\infty$-category admitting geometric realizations. Then $F$ can be extended to an $\infty$-sheaf $F_{ext}$ on $(N^D_{\bb}(\op{St}_{\Ca}), \op{\T-loc})$. \\
		
		In particular given any object $\X \in \op{St}_{\Ca}$ and an atlas $ x: X \to \X$ admitting $\T$-local sections, $F_{ext}(\X)$ can be computed as a limit over the \v{C}ech nerve $X^{+}_{\bb,x}$ over $x$. In other words, 
		\begin{equation}\label{keydefofFext} F_{ext}(\X) \cong \op{lim} (\begin{tikzcd}
				F(X) \arrow[r,shift left =2] \arrow[r,shift right =2]  & F(X \times_{\X} X)  \arrow[l,dotted] \arrow[r,shift left =4] \arrow [r, shift right =4] \arrow [r]   &  \arrow[l,shift left=2, dotted] \arrow[l, shift right=2,dotted]\cdots 
			\end{tikzcd} ).
		\end{equation}
		
	\end{theorem}

	\textbf{Idea of constructing the functor $F_{ext}$:}\label{Fextidea} 
	Given any zero simplex $\sigma_0$ i.e. an object $\X \in N^D_{\bb}(\op{St}_{\Ca})$, we would like to define $F_{ext}(\X)$ by \cref{keydefofFext}. As this definition depends on the atlas $x$, we start with an intrinsic description considering all \v{C}ech nerves of atlases of objects of $\op{St}_{\Ca}$.\\
	Let $\op{Cov}(\op{St}_{\Ca})$ be the subctegory of the $\infty$-category $\op{Fun}(N(\Delta_+)^{op},N^D_{\bb}(\op{St}_{\Ca}))$ spanned by objects which are \v{C}ech nerves of atlases admitting $\T$-local sections of objects of $\op{St}_{\Ca}$. We shall denote the objects of $\op{Cov}(\op{St}_{\Ca})$ by pairs $(\X,x: X \to \X)$ where $\X \in \op{St}_{\Ca}$ and $x: X \to \X$ is an atlas admitting $\T$-local sections.\\
	
	The inclusion $[-1] \hookrightarrow N(\Delta_+)$ induces the morphism 
	\[ p: \op{Cov}(\op{St}_{\Ca})^{op} \to N^D_{\bb}(\op{St}_{\Ca})^{op}. \]
	
	As every object in $\op{St}_{\Ca}$ admits a cover, the morphism $p$ is surjective on the level of objects.
	\begin{claim}\label{covnonempty}
		The morphism $p: \op{Cov}(\op{St}_{\Ca})  \to N^D_{\bb}(\op{St}_{\Ca})$ is surjective on $n$-simplicies. More precisely, let $\sigma_n$ be an $n$-simplex of $N^D_{\bb}(\op{St}_{\Ca})$ where $ n \ge 1$. Then there exists a map
		\[ \sigma^1_n: \Delta^1 \times \Delta^n \to N^D_{\bb}(\op{St}_{\Ca}) \] 
		such that 
		\begin{enumerate}
			\item $\sigma^1_n|_{[0] \times \Delta^n}$ factors through $N(\Ca) \subset N^D_{\bb}(\op{St}_{\Ca})$,
			\item $\sigma^1_n([1] \times \Delta^n) = \sigma_n$ and
			\item $\sigma^1_n(\Delta^1 \times [j]) $ is a morphism admitting $\T$-local sections for all $0 \le j \le n$.
		\end{enumerate}
	\end{claim}
	\begin{proof}[Proof of the claim]
		The case $n=1$ follows from \cref{covnonemptyclaim}. The general case follows by induction as for any $n$-simplex $(\X_0,\dots,\X_n)$, there exists compatible choice of atlas $(X_1,\dots,X_n)$ by induction. Applying the construction for $n=1$ to the atlas $X_0$ allows us to extend the family to $(X_0,X_1,\dots,X_n)$.\\
		
		This proves the surjectivity of $p$ on the level of simplices because the \v{C}ech nerve of $\sigma^1_n$ (considered as an edge $ \Delta^1 \to \op{Fun}(\Delta^n,N^D_{\bb}(\op{St}_{\Ca}))$) produces an element in $n$-simplex of $\op{Cov}(\op{St}_{\Ca})$.
	\end{proof}
	
	Morphisms of coverings induce morphisms of \v{C}ech nerves that are mapped to the identity via $p$. We shall denote
	the collection of all these morphisms in $\op{Cov}(\op{St}_{\Ca})$ by $R$. These are called \emph{refinements of coverings}. 
	
	\begin{proposition}\label{covislocalization}
		The morphism $p: \op{Cov}(\op{St}_{\Ca})^{op} \to N^D_{\bb}(\op{St}_{\Ca})^{op}$ is a localization of $\op{Cov}(\op{St}_{\Ca})$ along $R$. 
	\end{proposition}
	
	\begin{proof}
		As $p$ sends $R$ to equivalences, the morphism $p$ induces a morphism \[p': \op{Cov}(\op{St}_{\Ca})[R^{-1}]^{op}\to N^D_{\bb}(\op{St}_{\Ca})^{op}
		\]
		where $ i: \op{Cov}(\op{St}_{\Ca})^{op} \to \op{Cov}(\op{St}_{\Ca})^{op}[R^{-1}] $ is the anodyne map constructed in existence of localization (\cite[Proposition 6.3.2.1]{kerodon}). We want to show that $p'$ is a categorical equivalence. In particular we show that $p'$ is a trivial fibration of simplicial sets (which is a categorical equivalence by \cite[Proposition 23.11]{Rezk}).\\
		Thus given any commutative diagram of simplicial sets 
		\begin{equation}
			\begin{tikzcd}
				\partial\Delta^n \arrow[r,"\tau_n"] \arrow[d,hookrightarrow] & \op{Cov}(\op{St}_{\Ca})[R^{-1}]^{op} \arrow[d,"p'"] \\
				\Delta^n \arrow[ur,dotted,"\tau'_n"] \arrow[r,"\sigma_n"] & N^D_{\bb}(\op{St}_{\Ca})^{op},
			\end{tikzcd}
		\end{equation}
		we need to show the existence of a dotted arrow such that the diagram commutes. \\
		As the objects of $\op{Cov}(\op{St}_{\Ca})[R^{-1}]^{op}$ and $\op{Cov}(\op{St}_{\Ca})^{op}$ coincide and $p$ is surjective on $0$-simplices, this implies that $p'$ is surjective on $0$-simplices. This shows the claim for $n =0$. \\
		
		Let $n \ge 1$. We shall denote the vertices of $\sigma_n$ and $\tau_n$ by $\X_0,\X_1,\cdots \X_n$ and $(\X_0,x_0: X_0 \to \X),(\X_1,x_1: X_1 \to \X_1), \cdots (\X_n,x_n:X_n \to \X_n)$. As $p$ is surjective on $n$-simplices, there exists a morphism $\sigma'_n: \Delta^n \to \op{Cov}(\op{St}_{\Ca})^{op}$ which lifts $\sigma_n$. Let us denote the vertices of $\sigma'_n$ by $(\X_0,x_0': X_0' \to \X_0),(\X_1,x_1': \X_1' \to \X_1),\dots,(\X_n,x_n':X_n' \to \X_n)$.\\ 
		For each $0 \le i \le n$, the morphism $x''_i: X''_i:= X_i \times _{\X_i} X_i' \to \X_i$ is an atlas admitting $\T$-local sections. The morphisms $\sigma'_n$ and $\tau_n$ induces a morphism  \[\sigma''_n: \partial \Delta^n \to \op{Cov}(\op{St}_{\Ca})[R^{-1}]^{op} \] whose vertices are given by $(\X_0,x''_0),(\X_1,x''_i),\dots,(\X_n,x''_n)$. Note that the projection maps $\op{pr}_i:X_i'' \to X_i$ and $\op{pr}_i':X_i'' \to \X_i'$ are elements of $R$ and therefore become equivalences in the localization. This induces a map \[ f_n: \partial\Delta^n \times \Delta^1 \coprod_{\{0\} \times \partial\Delta^n} \{0\} \times\Delta^n  \to \op{Cov}(\op{St}_{\Ca})[R^{-1}]^{op} \]
		where $f_n|_{\{0\} \times \Delta^n} =\sigma'_n$ , $f_n|_{\{1\} \times \partial\Delta^n} = \sigma''_n$ and $f_n|_{[k] \times \Delta^1} = \op{pr}_k'$ for all $0 \le k \le n$. Applying \cref{isofibrationfunclift} to the morphism $f_n$ induces a morphism $f_n': \Delta^n \times \Delta^1 \to \op{Cov}(\op{St}_{\Ca})[R^{-1}]^{op}$. In particular the morphism $\sigma''_n$ extends to a morphism $\tau''_n: \Delta^n \to \op{Cov}(\op{St}_{\Ca})[R^{-1}]^{op}$. The morphisms $\tau''_n$ and $\tau_n$ produces a map  
		\[ g_n: \partial\Delta^n \times \Delta^1 \coprod_{\partial\Delta^n \times \{1\}}  \Delta^n \times \{1\} \to \op{Cov}(\op{St}_{\Ca})[R^{-1}]^{op} \] 
		where $g_n|_{\Delta^n \times \{1\}} = \tau''_n$, $g_n|_{\partial\Delta^n \times \{0\}} = \tau_n$ and $g_n|_{[k] \times \Delta^1} = \op{pr}_k$ for $ 0 \le k \le n$. As the morphism $\op{Fun}(\Delta^n,\D) \to \op{Fun}(\partial\Delta^n,\D)$ is an isofibration for $n \ge 1$,(\cite[Proposition 2.2.5]{land2021introduction}), $g_n$ extends to a morphism $g_n': \Delta^n \times \Delta^1 \to \op{Cov}(\op{St}_{\Ca})[R^{-1}]^{op}$. In particular, we have extended $\tau_n$ to a morphism $\tau'_n: \Delta^n \to \op{Cov}(\op{St}_{\Ca})[R^{-1}]^{op}$. Thus there exists a solution to the lifting problem. This shows that $p'$ is a trivial fibration.
	\end{proof}
	
	The fact that the morphism $p$ is a localization makes it easy to construct the extension $F_{ext}$ in \cref{thmmain} by first extending $F$ to the \v{C}ech nerves of coverings and notion that the sheaf condition implies that this induces a functor on the localization $\op{Cov}[R^{-1}]$. Let us explain this in detail.
	
	\begin{proof}[Proof of \cref{thmmain}]
		
		\begin{enumerate}
			
			\item  (\textbf{Constructing the functor $F_{ext}$})
			
			We define a morphism 
			\[ \phi: \op{Cov}(\op{St}_{\Ca})^{op} \xrightarrow{F} \op{Fun}(N(\Delta),\D) \xrightarrow{i} \op{Fun}(N(\Delta_+),\D) \xrightarrow{\op{res}|_{[-1]}} \D   \]
			as follows:
			\begin{enumerate}
				\item The map \[F: \op{Cov}(\op{St}_{\Ca})^{op} \to \op{Fun}(N(\Delta),\D) \] is the functor $F$ applied to the restricted simplicial object $X_{\bb,x}$ of an object $X^{+}_{\bb,x}$ of $\op{Cov}(\sigma_0)$. 
				\item To associate limit diagrams to cosimplicial objects, we apply \cite[Corollary 4.3.2.16]{HTT}. Let $\Ca^{(0)} = N(\Delta)$ and $\Ca = N(\Delta_+)$. As $\D$ admits geometric realizations, applying \cite[Corollary 4.3.2.16]{HTT}, we get a morphism 
				\[ i: \op{Fun}(N(\Delta),\D) \to \op{Fun}(N(\Delta_+),\D). \]
				On the level of objects, the morphism $i$ sends a cosimplicial object to an augmented cosimplicial object given by its limit diagram.
				
				\item The map 
				\[ \op{res}|_{[-1]}: \op{Fun}(N(\Delta_+),\D) \to \D  \]
				is induced by the inclusion map $[-1] \to \Delta_+$. On the level of objects, it sends an augmented cosimplicial object $Y^+_{\bb}$ to $Y_{-1}$.
			\end{enumerate}
			For any object $X^+_{\bb,x}$ in $\op{Cov}(\op{St}_{\Ca})$, we claim that $\phi(X^+_{\bb,x}) \cong \op{lim}_{\bb \in \Delta}F(X_{\bb,x})$.  This follows from the sheaf condition, namely let $X^+_{\bb,x}$ and $X^{'+}_{\bb,x}$ be two \v{C}ech nerves of two atlases $x: X \to \X$ and $x': X \to \X$. Applying \cref{equivprodlemma} to the pullback square
			\begin{equation}
				\begin{tikzcd}
					X \times_{\X} X' \arrow[r] \arrow[d] & X \arrow[d,"x"] \\
					X' \arrow[r,"x'"] & \X,
				\end{tikzcd}
			\end{equation}
			we get that the morphisms $\phi(X^+_{\bb,x}) \to \phi(X^+_{\bb,x} \times X^{'+}_{\bb,x})$ and $\phi(X^+_{\bb,x}) \to \phi(X^+_{\bb,x}\times X^{'+}_{\bb,x})$ are equivalences because $F$ satisfies descent along morphisms admitting $\T$-local sections (\cref{covlocsecdesc}). Morever if $f: (\X,x: X \to \X) \to (\X, x': X' \to \X)$ is a morphism in $\op{Cov}(\op{St}_{\Ca})$, then the above argument shows us that $\phi$ sends $f$ to equivalences. Thus $\phi$ maps every element of $R$ to equivalences. As the morphism $p$ is a localization (\cref{covislocalization}), there exists a functor  
			\[ F_{ext}: N^D_{\bb}(\op{St}_{\Ca})^{op} \to \D \] such that $\phi \cong F_{ext} \circ p$ in $\op{Fun}(\op{Cov}(\op{St}_{\Ca})^{op},\D)$. The equivalence $\phi \circ F_{ext} \circ p$ gives us that $F_{ext}(\X)$ can be computed as a simplicial limit over any \v{C}ech cover of an atlas admitting $\T$-local sections. 
			\item (\textbf{$F_{ext}$ is an  $\infty$-sheaf}) For showing that $F_{ext}$ is an $\infty$-sheaf, we need to check that for any morphism $p: \Y \to \X$  admitting $\T$-local sections, $p$ satisfies $F_{ext}$-descent. Thus we want to show that
			\[ F_{ext}(\X) \cong \op{lim}_{n\in \Delta} \F_{ext}(X^+_{\bb,p}).   \]
			As $p$ admits $\T$-local sections, there exists a commutative diagram	
			\begin{center}
				\begin{tikzcd}
					Y \arrow[r,"q'"] \arrow[d,"p'"] & \Y \arrow[d,"p"] \\
					X \arrow[r,"q"] & \X
				\end{tikzcd}
			\end{center}
			
			where $p'$ admits $\T$-local sections and $q',q'$ are atlases admitting $\T$-local sections (\cref{tlocalsectionanotherformulation}). By assumption, $p'$ satisfies $F_{ext}$-descent. Also $q$ and $q'$ satisfy $F_{ext}$-descent by definition of the functor $F_{ext}$.  Then applying \cref{equivprodlemma}, we get that $p$ satisfies $F_{ext}$-descent. This completes the proof.
		\end{enumerate}
	\end{proof}

	\subsection{The motivic stable homotopy category of an algebraic stack.}
	
	In \cite{robalo2013noncommutative}, Robalo explains that the construction of motivic stable homotopy theory (\cite{MorVod}) can be viewed as a functor taking values in $\infty$-categories.\\
	
	Let $\Schf$ denote the category of Noetherian schemes with finite Krull dimension. Robalo uses the construction of $\SH$ to define a functor 
	\begin{equation}
		\Sho(-): \Schf^{op}\to \op{CAlg}(\op{Pr}^L_{stb}). ~~~ \cite[\text{Section 9.1}]{Robalothesis}
	\end{equation}
	Let us unravel the information contained in this functor. As the classical $\op{SH}$ admits a symmetric monoidal structure, the $\infty$-category $\Sho(S)$ is a symmetric monoidal $\infty$-category, i.e. it comes equipped with a coCartesian fibration (\cite[Definition 2.4.2.1]{HTT}) $p_S:\Sho(S) \to N(\op{Fin}_*)$  such that $\Sho(S)_{\langle n \rangle} \cong \Sho(S)^{\times n}_{\langle 1 \rangle}$. Denote $\SH(S):= \Sho(S)_{\langle 1 \rangle}$. The coCartesian fibration encodes the symmetric monoidal structure of the $\infty$- category $\SH(S)$ in a coherent way.\\
	The $\infty$-category $\SH(S)$ is also presentable as it arises from localization of presheaves of smooth schemes over $S$ (\cite[Theorem 5.5.1.1]{HTT}). It is also a stable $\infty$-category (\cite[Definition 1.1.1.9]{HA}) as it inherits the triangulated structure of $\op{SH}$ constructed in \cite{MorVod}. As pullback morphism for a morphism $f$ in $\Schf$ is colimit preserving, the $\infty$-category $\Sho(S)$ lands in the $\infty$-category of presentable stable symmetric monoidal $\infty$-categories which is denoted by $\op{CAlg}(\op{Pr}^L_{stb})$. \\
	The  $\infty$-category $\SH(S)$  is the called the stable motivic homotopy category of $S$. \\
	As explained in \cite[Remark 9.3.1]{Robalothesis}, the stable motivic homotopy theory can be extended to all schemes.
	The functor $\Sho(-)$ is a Nisnevich sheaf (\cite[Proposition 6.24]{Equsixop}). 
	\begin{corollary}\label{Shstckdef}
		The functor $\Sho(-)$ extends to an $\infty$-sheaf \[\Shext(-): \Nst^{op} \to \op{CAlg}(\op{Pr}^L_{stb}). \] \\
		
		Morever, for any algebraic stack $\X \in \Nst$ that admits a schematic atlas \linebreak $x: X \to \X$, one has
		
		\begin{equation}
			\Shext(\X) \cong \op{lim}\Big(\begin{tikzcd}
				\Sho(X) \arrow[r, shift right =2] \arrow[r,shift left =2] & \Sho(X \times_{\X} X) \arrow[l,dotted] \arrow [r,shift right =4] \arrow[r] \arrow[r,shift left =4] & \cdots \arrow[l, shift left =2 ,dotted] \arrow[l, shift right =2, dotted]
			\end{tikzcd} \Big)
		\end{equation}
		where the limit is over the \v{C}ech nerve of $x$.

	\end{corollary} 
	\begin{proof}
	As $\op{CAlg}(\op{Pr}^L_{stb})$ admit small limits (\cite[Proposition 3.2.2.1]{HA}), we can apply \cref{thmmain} to the functor $\Sho(-)$ with $\T = \op{Nis}$ and $N^D_{\bb}(\op{St}_{\Ca}) = \Nstalg$. This gives us an $\infty$-sheaf  $\SAlg(-): \Nstalg^{op} \to \op{CAlg}(\op{Pr}^L_{stb})$. \\
		Applying the theorem again to the $\infty$-sheaf $\SAlg(-)$ with $N(\Ca) = \Nstalg$, $\T=\op{Nis}$ and $N^D_{\bb}(\op{St}_{\Ca})=\Nst$, one gets an $\infty$-sheaf $\Shext(-)$. \\
		The limit description is a consequence of \cref{keydefofFext} applied to the functor $\Shext(\X)$. 
		
	\end{proof}
	\begin{notation}
		For any algebraic stack $\X \in \op{Nis-locSt}$, we shall denote the underlying presentable stable $\infty$-category of the symmetric monoidal $\infty$-category $\Shext(\X)$ by $\SHe(\X)$. We shall call $\SHe(\X)$ to be the \textit{stable motivic homotopy category of $\X$}.
	\end{notation}
	\begin{remark}
	\begin{enumerate}
	    \item $\Shext(\X)$ can also be computed as a limit over the semisimplicial category $\Delta_s$ as $\Delta_s \hookrightarrow \Delta$ is cofinal.
	    \item 	Recall from \cref{quotstackinnst} that for quotient stacks $[X/G]$, the atlas $X \to [X/G]$ does not admit Nisnevich-local sections and thus we need to replace it by $X \times^G \op{GL}_n$ to compute $\Shext$.\\ 
		In \cite{Equsixop}, Hoyois defines $SH$ for global quotient stacks by tame reductive groups. His construction a priori may depend on choice of presentation of the quotient stack. Our construction has the advantage that it is independent of such a choice and it morever allows us to drop the tameness assumption.
	\end{enumerate}
     
\end{remark}
In \cite[Section 12]{khan2021generalized}, Khan and Ravi introduced the notion of limit-extended functor $\mathbf{SH}_{\triangleleft}(-)$ for derived algebraic stacks.
For any $\X \in \op{Nis-locSt}$, let $\op{Lis_{\X}}$ be the $\infty$-category of pairs $(t,T)$ where $T$ is a scheme and $t: T \to \X$. The functor $\mathbf{SH}_{\triangleleft}(\X)$ ((\cite[Construction 12.1]{khan2021generalized}) is defined as 
 \begin{equation}
     \mathbf{SH}_{\triangleleft}(\X):= \varprojlim_{(t,T) \in \op{Lis}_{\X}}\Sho(T).
 \end{equation}
We have a canonical functor $\mathbf{SH}_{\triangleleft}(\X) \to \Shext(\X)$ which is induced by restricting it to the subcategory $\op{Lis}_{X^+_{\bb,x},\X}$ of $\op{Lis}_{\X}$ consisting of objects of semisimplicial \v{C}ech nerve of $x: X \to \X$ where $x$ is an atlas admitting Nisnevich-local sections.\\  
The following compares $\Shext(\X)$ to their construction.  
\begin{corollary}\label{limitextendedcomparison}
    The canonical functor 
    \begin{equation}
        \mathbf{SH}_{\triangleleft}(\X) \to \Shext(\X)
    \end{equation}
    is an equivalence.
\end{corollary}
\begin{proof} 
    As $\Sho$ satisfies descent along morphism of schemes admitting Nisnevich-local sections, we have 
     \begin{equation}
         \mathbf{SH}_{\triangleleft}(\X) \cong \varprojlim_{(t,T) \in \op{Lis}_{\X}} \varprojlim_{P^n \in \op{Lis}_{P^+_{\bb,t_x},T}} \Sho(P^n)
      \end{equation}
      where $P^n$ is the fiber product of $X \times_{\X} T$ over $T$  $n$ times and $t_x$ is the projection map $t_x : X \times_{\X} T \to T$. Note that the morphism $t_x$ is a morphism of schemes which admits Nisnevich-local sections by the property of map $x$.\\
      We have pullback maps $\Sho(X^n_{\X}) \to \Sho(P^n)$ for every $n$. This induces a map 
      \begin{equation}
          \Shext(\X) \cong \varprojlim_{X^n_{\X} \in \op{Lis}_{X^+_{\bb,x}}} \Sho(X^n_{\X}) \to \varprojlim_{(t,T)\in \op{Lis}_{\X}} \varprojlim_{P^n \in \op{Lis}_{P^+_{\bb,t_x},T}} \Sho(P^n).
      \end{equation}
      This gives us a functor 
      \begin{equation}
          \Shext(\X) \to \mathbf{SH}_{\triangleleft}(\X).
      \end{equation}
      By construction, it follows that this is the inverse of the canonical map. Hence it is an equivalence.
\end{proof}

	The description of the $\infty$-sheaf $\Shext(-)$ gives us the following functors.
	\begin{notation}
		\begin{enumerate}
			\item 	Let $f: \X \to \Y$ be a morphism in $\Nst$. We denote the \textit{pullback functor} $\Shext(f): \Shext(\Y) \to \Shext(\X)$ by $f^{*\otimes} $. We shall also write $f^*: \SHe(\Y) \to 
			\SHe(\X)$ as the functor $f^{*\otimes}$ on the level of underlying $\infty$-categories. As $f^*$ is a colimit preserving functor, the adjoint functor theorem ( \cite[Corollay 5.5.2.9]{HTT}) says there exists a right adjoint  \[f_*: \SHe(X) \to \SHe(Y) \] which we call the \textit{pushforward functor}. 
			\item  	As $\Shext(\X)$ is a symmetric monoidal $\infty$-category, we shall denote the functor induced by the symmetric monoidal structure by \[ -\otimes -: \SHe(\X) \times \SHe(\X) \to \SHe(\X). \]
		\end{enumerate}

	\end{notation}
	
	For a scheme $X$, the $\infty$-category $\Sho(X)$ is closed. Let us explain this notion briefly.\\
	Given any two objects $E$ and $E'$ in $\SH(X)$, one has objects $\op{Hom}(E,E')$ and $\op{Hom}(E',E)$ in $\SH(X)$ with maps $\op{Hom}(E,E') \otimes E \to E'$ and $\op{Hom}(E',E) \otimes E' \to E$ satisfying usual universal properties. In other words, the tensor product realized as a functor $\SH(X) \to \op{Fun}(\SH(X),\SH(X))$ factorizes via $\op{Fun}^L(\SH(X),\SH(X))$ (\cite[Definition 4.1.15]{HA}). We have the following proposition.
	
	\begin{proposition}\label{closedunderlim}\cite[Remark 1.5.3]{liu2017enhanced}
		For any $\X \in \op{Nis-locSt}$, the $\infty$-category $\Shext(\X)$ is closed.
	\end{proposition}
	
	\begin{proof}
		Let $x:X \to \X$ be an atlas admitting Nisnevich-local sections. Then we have a functor \[p^{\otimes}: N(\Delta) \to \op{CAlg}(\op{Pr}^L_{stb}) \] induced by the \v{C}ech nerve of $x$. As $\Sho(X^n_{\X})$ is closed for every $n$. Then by \cite[Remark 1.5.3]{liu2017enhanced}, we get that the limit of $p^{\otimes}$ i.e. $\Shext(\X)$ is closed.
	\end{proof}
	\begin{notation}
		For any objects $\E,\E' \in \SHe(\X)$, we shall denote $\op{Hom}_{\SHe(\X)}(\E,\E')$ to be the internal Hom. 
	\end{notation}
	
	Thus we have defined four functors $f^*,f_*, - \otimes-$ and $\op{Hom}(-,-)$ along with the functor $\Shext(-)$. 
		\section{Enhanced operations for stable homotopy theory of algebraic stacks}
	In the previous section, we have extended the stable homotopy functor $\SH$ from schemes to algebraic stacks. We have also defined the four functors $f^*,f_*, - \otimes-$ and $ \op{Hom}(-,-)$. The goal of this chapter is to construct the functors $f_!$, $f^!$ and prove the base change and projection formula (\cref{thmmain2}).\\
	
	The key idea is to construct these functors and proving the above mentioned properties via the enhanced operation map due to Liu and Zheng (\cite{liu2017enhanced}). The enhanced operation map is a functor which encodes all of this information.
	As $\SH$ for schemes satisfy relations among six operations, the enhanced operation map can be constructed on the level of schemes (see \cite[Section 9.4]{Robalothesis}). We shall extend the enhanced operation map from schemes to algebraic stacks which shall prove \cref{thmmain2}. \\

	\subsection{Statement of the theorem and motivation for enhanced operation map.}
	The extraordinary pushforward $f_!$ and extraordinary pull-back functors $f^!$ are defined for morphisms of schemes that are separated and of finite type (\cite[Theorem 9.4.8]{Robalothesis}).  We will denote by $\Schf' \subset \Schf$ the category of schemes in which morphisms are separated and of finite type. With this notation, the functors $f_!$ and $f^!$ can be assembled into functors \[\SH_!: N(\Schf') \to \op{Pr}^L_{stb} ~~,~~ \SH^!: N(\Schf')^{op} \to \op{Pr}^R_{stb}.\] 
	We shall denote by $\op{Nis-locSt}' \subset\op{Nis-locSt}$ the subcategory in which morphisms are representable and separated of finite type. Note that for a representable morphism, separated is equivalent to the fact that the diagonal is a closed immersion (\cite[\href{https://stacks.math.columbia.edu/tag/04YS}{Tag 04YS}]{stacks-project}).
	
	\begin{theorem}\label{thmmain2}
		
		The functors $\SH_!$ and $\SH^!$  extend to functors
		\[ \Sheshr: N^D_{\bb}(\op{Nis-locSt}') \to \op{Pr}^L_{stb}   \]
		and \[ \Sheshrup: N^D_{\bb}(\op{Nis-locSt}')^{op} \to \op{Pr}^R_{stb}.  \]  These functors satisfy:
		\begin{enumerate}
			
			\item(\textbf{Base change}) Let 
			\begin{equation}
				\begin{tikzcd}
					\X' \arrow[r,"g'"] \arrow[d,"f'"] & \X \arrow[d,"f"] \\
					\Y' \arrow[r,"g"] & \Y
				\end{tikzcd}
			\end{equation}
			be a pullback diagram in $\op{Nis-locSt}$ where $f$ and $f'$ are separated of finite type. Then the diagram
			\begin{equation}
				\begin{tikzcd}
					\SHe(\X) \arrow[r,"g^{'*}"]  \arrow[d,"f_!"] & \SHe(\X') \arrow[d,"f'_!"] \\
					\SHe(\Y) \arrow[r,"g*"] & \SHe(\Y')
				\end{tikzcd}
			\end{equation}
			commutes in $\op{CAlg}(\op{Pr}^L_{stb})$. In other words we have an equivalence of functors  \[ \op{Ex}(\Delta^*_{\#}):g^*\circ f_! \cong f'_! \circ g'^*  \] in the functor category $\op{Fun}(\SHe(\X),\SHe(\Y'))$.
			\item (\textbf{Projection formula}) Let $f: \X \to \Y$ be a morphism in $\op{Nis-locSt}'$. Given $E \in \SH(\X)$ and $E' \in \SH(\Y)$, there exists an equivalence
			\begin{equation}
				f_!(E \otimes f^*(E')) \cong f_!(E) \otimes E'.
			\end{equation}
		\end{enumerate}
		
	\end{theorem}
	\begin{remark}
		To prove the theorem, it suffices to construct the functor $\Sheshr$. Given $\Sheshr$, the functor $\Sheshrup$ can be defined by $\Sheshrup = (\Sheshr)^{op}: \Nst^{op} \to (\op{Pr}^L_{stb})^{op} \cong \op{Pr}^R_{stb}$. The equivalence $ (\op{Pr}^L_{stb})^{op} \cong \op{Pr}^R_{stb}$ follows from \cite[Corollary 5.5.3.4]{HTT}. The equivalence relies on the fact that colimit-preserving between presentable $\infty$-categories admit right adjoints. Thus for $f: \X \to \Y$ a representable morphism of stacks which is separated and of finite type, $f^!:= \Sheshrup(f): \SHe(\Y) \to \SHe(\X)$ is right adjoint to $f_!:=\Sheshr(f): \SHe(\X) \to \SHe(\Y)$.\\
	\end{remark}
	We shall prove this theorem in \cref{proofofthmmain2}. The proof of the theorem relies on extending a special kind of map from schemes to algebraic stacks. We call this map the \emph{"enhanced operation map"} due to Liu and Zheng (\cite{liu2017enhanced}). The existence of the enhanced operation map shall give us the lower shriek functors, base change formula and projection formula. We now try to motivate the source and target of this map. \\
	
	Let us first motivate the target. This comes from the projection formula. The projection formula can viewed as a property in the $\infty$-category of module objects $\op{Mod}(\op{Pr}^L_{stb})$. Let us briefly recall the $\infty$-categorical analog of classical module categories i.e. $\op{Mod}(\op{Pr}^L_{stb})$ (see \cite[Section 9.4.1.2]{Robalothesis} for precise definition).  \\
	
	The objects of $\op{Mod}(\op{Pr}^L_{stb})$ are pairs $(\Ca^{\otimes},\M)$ where $\Ca^{\otimes}$ is a symmetric monoidal $\infty$-category (i.e. an object in $\op{CAlg}(\op{Pr}^L_{stb})$) and $\M$ is a presentable stable $\infty$-category with a morphism $\Ca \times \M \to \M$ which incorporates the module structure of the object $\M$ (here $\Ca$ is the underlying $\infty$-category of $\Ca^{\otimes}$). A morphism $ (\Ca^{\otimes},\M ) \to (\Ca^{'\otimes},\M')$  in $\op{Mod}(\op{Pr}^L_{stb})$ consists of morphisms of commutative algebra objects $ u: \Ca^{\otimes} \to \Ca^{'\otimes}$ and a morphism $v: \M \to \M'$ in $\op{Pr}^L_{stb}$ which is $\Ca$-linear where $\M'$ is endowed with a $\Ca^{\otimes}$-module structure via $u$. In particular for objects $c \in \Ca$ and $m \in M$, one has an equivalence 
	\begin{equation}\label{modhomomorphism}
		v(c \otimes m) \cong u(c) \otimes v(m).
	\end{equation}

	In the context of stable homotopy theory of schemes, a morphism $f: X \to Y$ induces a monoidal pullback functor $f^{*\otimes}: \Sho(Y) \to \Sho(X)$. Then the pair $(\Sho(Y),\SH(X))$ is an example of a module object. We can visualize this as an $\infty$-categorical generalization of the statement that a morphism $ g: A \to B$ of rings makes $B$ an $A$-module. Also for any object $Z \in \Schf$, the pair $(\Sho(Z),\SH(Z))$ is an object in $\op{Mod}(\op{Pr}^L_{stb})$ (the module structure is induced by the tensor product). \\
	
	The projection formula is equivalent to the statement that the pair of morphisms
	\begin{equation}
		(\op{id},f_!):(\Sho(Y),\SH(X)) \to (\Sho(Y),\SH(Y))
	\end{equation}
	is a morphism in $\op{Mod}(\op{Pr}^L_{stb})$. This follows from the condition of module morphism (\cref{modhomomorphism} applied to $v= f_!$ and $ c = \op{id}$) in this context. The above discussion motivates that $\op{Mod}(\op{Pr}^L_{stb})$ will be the target of the enhanced operation map.\\
	
	Let us motivate the source of the enhanced operation map. As stated before, the map encodes both the lower shriek functors and the base change formula. To combine these, one defines a  simplicial set where the $1$-simplices are cartesian squares
	\begin{equation}
		\begin{tikzcd}
			\X' \arrow[r,"g'"] \arrow[d,"f'"] & \X \arrow[d,"f"] \\
			\Y' \arrow[r,"g"] & \Y
		\end{tikzcd}
	\end{equation}
	where $f$ (and thus $f'$) is representable, separated and finite type.\\
	
	Thus the enhanced operation map is a functor from a simplicial set whose simplices consists of pullback squares as above and takes values in $\op{Mod}(\op{Pr}^L_{stb})$.\\
	
	The source of the enhanced operation map naturally defines a bisimplicial set in which vertical arrows are separated and finite type, which is called a bi-marked simplicial set (\cite[Definition 3.9]{Gluerestnerv}). As the functor $f_!$ is usually constructed by combining  constructions for open embeddings and proper morphisms, it will be useful to have a more general notion of a multi-simplicial set in which the class of arrows in some directions are restricted. Let us therefore recall these notions from the article of Liu and Zheng (\cite[Section 3]{Gluerestnerv}).

	\subsection{Multisimplicial, multi-marked and multi-tiled simplicial sets.}

	Let $I$ be a finite set and consider it as a discrete category.
	\begin{definition}\cite[Definition 3.1]{Gluerestnerv}
		An \textit{I-simplicial set} is a functor:
		\[ \op{Fun}(I,\Delta)^{op}=  (\underbrace{\Delta \times \Delta  \cdots \Delta}_{\text{I-times}})^{op} \to \Set  \]
		We denote the category of $I$-simplicial sets by $\Set_{I\Delta}$. If $I=\{1,2,\cdots,k\}$, then we denote it by $\op{Set}_{k\Delta}$.
	\end{definition}
	\begin{remark}
		By definition, $\Set_{1\Delta}=\Set_{\Delta}$ and similarly $\Set_{2\Delta}$ is the category of bisimplicial sets.
	\end{remark}
	\begin{notation}
		We shall denote any object $(n_i)_{i \in I} $ of $\op{Fun}(I,\Delta)$ by $\niI$
		We denote $\Delta^{\niI}$ to be the $I$-simplicial set represented by $\prod_i\Delta^{n_i}$. For an $I$-simplicial set, we denote $S_{\niI}$ by $S(\niI)$. 
	\end{notation}
	
	We discuss adjunctions between $\Set_{k\Delta}$ and $\Set_{\Delta}$. 
	
	\begin{notation}\cite[Definition 3.3]{Gluerestnerv}
		\begin{enumerate}
			\item Denote $f: I=\{1,2,\cdots,k\}\to \{1\}$ be the projection map. This induces the functor $\Delta \to \op{Fun}(I,\Delta)$ which induces the \textit{diagonal functor}:
			\[ \delta^*_k:\Set_{k\Delta} \to \Set_{\Delta}  \] 
			which takes an $k$-simplicial set $S$ to $\delta^*_k(S)$ which evaluated on $[n]$ is $ S([n],[n],\cdots [n])$. \\
			This functor has a right adjoint:
			\[\delta^k_*:\Set_{\Delta} \to \Set_{k\Delta}  \]
			which evaluated on $S$, defines a $k$-simplicial set defined as 
			\[ \delta^k_*(S)_{\niI} = \op{Hom}_{\Set_{\Delta}}(\prod_{i \in I} \Delta^{n_i},S) \]
			\item  Similarly an injection of sets $f: J \hookrightarrow I$ induces a functor $(\Delta_f)^*: \Set_{J\Delta} \to \Set_{I\Delta}$ induced from $f$. It has a right adjoint, which we denote by 
			\[ \ep^I_J: \Set_{I\Delta} \to \Set_{J\Delta} \]
			defined by 
			\[ \ep^I_J(S)(\niI) = S((\niI,0)) \]
			where  we write $(\niI,0)$ for the vector with entries $0$ for $i \neq j$. We call the map $\ep^I_J$ as \textit{restriction functor}. \\
			If $I=\{1,2,\cdots,k\}$ and $J =\{j\}$, then we denote it by $\ep^k_j$. 
			\item Given $I=\{1,2,\cdots,k\}$ and $J \subset I$. We have the \textit{partial opposite} functor \[\op{op}^I_J: \Set_{k\Delta} \to \Set_{k\Delta} \] defined by taking opposite edges along the directions $j \in J$.  Using this notion, we define the \textit{twisted diagonal } functor as  \[ \dd^*_{k,J}:= \dd^*_I \circ \op{op}^I_J: \Set_{k\Delta} \to \Set_{\Delta}. \]
		\end{enumerate}
	\end{notation}
	\begin{example}
		\begin{enumerate}
			\item 	The map $\dd^2_*$ takes a simplicial set $S$ to the bisimplicial set $\dd^2_*S$ whose $(n_1,n_2)$ simplices are $\op{Hom}_{\sset}(\Delta^{n_1} \times \Delta^{n_2},S)$. If $S = N(C)$ where $C$ is an ordinary category, then these are just $n_1 \times n_2$ grids in $\Ca$. \\
			The map $\dd^*_2$ takes a bisimplicial set to its diagonal simplicial set. For $S =\dd^2_*N(C)$, the $n$-simplices of the simplicial set  $\dd^*_2(\dd^2_*N(C))$ are morphisms $\Delta^n \times \Delta^n \to N(C)$ (in other words these are $n \times n$ grids in $C$).\\
			\item 	The maps $\ep^2_1$ and $\ep^2_2$ send a bisimplicial set $S':(\Delta \times \Delta)^{op} \to \Set $ to the simplicial sets $S'|_{\Delta \times [0]}$ and $S'|_{[0] \times \Delta}$ respectively, i.e.. these are the restrictions to the first row and column of the bisimplicial set.
			
			\item For $k=1$, the twisted diagonal functor sends a simplicial set $S$ to $S^{op}$.For $k=2$, the partial opposite functor $\op{op}^2_{\{1\}}$ takes a bisimplicial set $S$ and sends to the bisimplicial set $S'$ which when restricted to direction $1$ gives the simplicial set $(\ep^2_1 S)^{op}$ and when restricted to direction $2$ gives the simplicial set $\ep^2_2S$.  In order to understand it more clearly, let us consider the bisimplicial set $\dd^2_*N(C)$. Then the $n$ simplices of the simplicial se $\dd^*_{2,\{1\}}(\dd^2_{*}N(C))$ are given by $n \times n$ grids $(\Delta^n)^{op} \times \Delta^n \to N(C)$.
			
		\end{enumerate}
	\end{example}

	\begin{definition}\cite[Definition 3.9]{Gluerestnerv}
		An \textit{$I$-marked simplicial set} is the data $(S,\E:=\{\E_i\}_{i \in I})$ where $S$ is a simplicial set and $\E$ is a set of edges $\E_i$ containing every degenerate edge of $S$. A morphism between $I$-marked simplicial sets $(S,\E)$ and $ (S',\E')$ is a morphism of simplicial sets $ f: X \to X'$ with the property $f(\E_i) \subset \E'_i$. We denote the category of $I$-marked simplicial sets as $\Set^{I+}_{\Delta}$. If $I = \{1,2,\cdots, k\}$, we denote the category of $I$-marked simplicial sets by $\Set^{k+}_{\Delta}$. 
	\end{definition}
	\begin{remark}
		An $I$-marked simplicial set is said to be an $I$-marked $\infty$-category if the underlying simplicial set is an $\infty$-category. \\
		For $k=1$, we get the notion of marked simplicial sets defined in \cite[Section 3.1]{HA}. 
	\end{remark}
	
	\begin{notation}
		
		\begin{enumerate}
			\item Given any $I$-simplicial set $S$, we can define an $I$-marked simplicial set $\dd^*_{I+}(S') = (\dd^*_I S', \E = \{ (\epsilon^I_iS)_1\}_{i \in I})$. When $k=2$, the marked simplicial set $\dd^*_{+2}(S')$ consists of the diagonal simplicial set of $S'$ with the marked edges being the edges of the simplicial set of the first row and first column of the bisimplicial set.
			\item Given any $I$-marked simplicial set $(S,\E)$, we can define an $I$-simplicial set $\dd^{I+}_*(S,\E)$ as the sub $I$-simplicial set of $\dd^I_*S$ which consists only of edges $\E_i$ in simplicial set $\ep^I_i(S)$. 
		\end{enumerate}
		
	\end{notation}
	This notion yields us to define the notion of restricted simplicial nerve. 
	\begin{definition}\cite[Definition 3.10]{Gluerestnerv}
		Let $(S,\E)$ be an $I$-marked simplicial set, then we define the \textit{restricted $I$-simplicial nerve} as 
		\[ S_{\E}:= \dd^{I+}_* (S,\E) \]
	\end{definition}
	\begin{example}
		Let $(S,\E) = (N(\Sc), \{P,O\})$ where $P$ and $O$ are the set of proper morphisms and open immersions respectively. Then $S_{\E}$ is the bisimplicial subset of the bisimplicial set $\dd^2_*N(\Sc)$ which consists of only proper morphisms as edges in the simplicial set $\ep^2_1(\dd^2_*N(\Sc))$ and open immersions as edges in the simplicial set $\ep^2_2(\dd^2_*N(\Sc))$.
	\end{example}
	\begin{definition}\cite[Definition 3.12]{Gluerestnerv}
		An \textit{I-tiled simplicial set} is the data $(S,\E=\{\E_i\}_{i\in I},\Q=\{\Q_{ij}\}_{i,j \in I, i \neq j})$ where $(X,\E)$ is a marked simplicial set and $\Q$ is a collection of set of squares $\Q_{ij}$ (i.e. $\Delta^1 \times \Delta^1 \to S$) such that 
		\begin{enumerate}
			\item the set of squares $\Q_{ij}$ and $\Q_{ji}$ are obtained from each other by transposition.
			\item The vertical arrows of each square in $\Q_{ij}$ are in $\E_i$ and the horizontal arrows are in $\E_j$.
			\item To every edge in $\E_i$, there is a square in $\Q_{ij}$ induced by the map $\op{id} \times s^0_0$. 
		\end{enumerate} A morphism of $I$-tiled simplicial sets $ f: (S,\E,\Q) \to (S',\E',\Q')$ which maps $f(\E_i) \subset \E'_i$ and $f(\Q_{ij}) \subset \Q'_{ij}$. We denote the category of $I$-tiled simplicial sets by $\Set^{I\bx}_{\Delta}$.
	\end{definition}
	
	\begin{notation}\cite[Remark 3.13]{Gluerestnerv}
		
		\begin{enumerate}
			\item Given any $I$-simplicial set $S$, we define an $I$-tiled simplicial set $\dd^*_{I\bx} (S):= (\dd^*_I S, \E,\Q)$ where $\E=\{\E_i= (\epsilon^I_i S')_1\}_{i\in I}$ and $\Q = \{\Q_{ij}= \op{Hom}(\Delta^1 \times \Delta^1,\dd^*_2\epsilon^I_{i,j}(S))\}_{i,j \in I, i \neq j}$. 
			\item Given any $I$-tiled simplicial set $(S',\E',\Q')$, we define an $I$-simplicial set $\dd^{I\bx}_*((S',\E',\Q'))$ as the  $I$-simplicial subset of $\dd^{I+}_*(S',\E')$ such that for $j,k \in I$ and $j \neq k$, every square in the simplicial set $\dd^*_2\ep^I_{jk}(\dd^{I+}_*(S',\E'))$ associated to any $(1,1)$-simplex lies in $\Q_{jk}$.
		\end{enumerate}
	\end{notation}
	\begin{remark}
	  Let $S$ be a bisimplicial set. Given any $(1,1)$-simplex of $S$, we can define a square in the diagonal simplicial set $\dd^*_2S$ as follows. A $(1,1)$-simplex corresponds to a morphism $\tau: \Delta^{(1,1)} \to S$. Applying the functor $\dd^*_2 (-)$, we get a morphism \[ \dd^*_2(\tau): \Delta^1 \times \Delta^1 \to \dd^*_2S. \] \\
		If $S = N(\Sc')_{P,O}$, then a square in $S$ corresponds to a morphism $ \Delta^1 \times \Delta^1 \to N(\Sc')$ where horizontal arrows are proper and vertical arrows are open. 
	\end{remark}
	\begin{definition}\cite[Definition 3.16]{Gluerestnerv}
		Let $\Ca$ be and $\infty$-category and $\E_1,\E_2$ be set of edges, denote $\E_1 \star^{\op{cart}} \E_2$ be the set of Cartesian squares. For an $I$-marked $\infty$-category $(\Ca,\E:=\{\E_i\}_{i \in I})$, we denote $\E^c_{ij}:= \E_i \star^{\op{cart}} \E_j$. Denote $ (\Ca,\E,\E^c)$ to be the $I$-tiled $\infty$-category. We define the \textit{Cartesian $I$-simplicial nerve} to be the $I$-simplicial set 
		\[ \Ca^{\op{cart}}_{\E}:= \dd^{I\bx}_*((\Ca,\E,\E^c)) \] 
	\end{definition}
	
	\begin{example}
		The bisimplicial set $N(\Sc)^{\op{cart}}_{P,O}$ is the sub-bisimplicial set of $\dd^2_*N(\Sc)$ which consists of proper morphisms as edges in one direction, open immersions as edges in other and every square formed by open and proper morphisms is a pullback square. \\
		
		Let us understand the simplicial set $\dd^*_{k}N(C)^{\op{cart}}_{\E}$ which will be the source of the enhanced operation map. A $n$-simplex of  $\dd^*_kN(C)^{\op{cart}}_{\E}$ is a morphism $\sigma_n:\underbrace{\Delta^n \times \Delta^n \cdots \times \Delta^n}_{ \op{k-times}} \to N(C)$  such that every edge  $\sigma_n|_i: \Delta^1 \to N(C)$ in direction $i$ lies in $\E_i$ for every $ i \in I$ and for every $ j \neq j' \in I$, the square $\sigma_n|_{j,j'}: \Delta^1 \times \Delta^1 \to N(C)$ is a pullback square formed by edges $\E_j$ and $\E_{j'}$. In case $ k=2$ and $(C,\E) = (\Sc,\{P,O\})$, the $n$-simplices of $\dd^*_2N(\op{Sch}')^{cart}_{P,O}$ are $n \times n$ grids of the form 
		
		\begin{equation}
			\begin{tikzcd}
				X_{00} \arrow[r] \arrow[d] & X_{01} \arrow[r] \arrow[d] & \cdots & X_{0n} \arrow[d] \\
				X_{10} \arrow[r]  & X_{11} \arrow[r] & \cdots & X_{1n} \\
				\vdots & \vdots & \vdots & \vdots \\
				X_{n1} \arrow[r] & X_{n2} & \cdots & X_{nn}
			\end{tikzcd}
		\end{equation}
		where vertical arrows are open, horizontal arrows are proper and each square is a pullback square. 
	\end{example}
	\begin{notation}\label{directionmapnotation}
		Let $(N(C),\{\E_1,\E_2\})$ be a marked $\infty$-category. For any $i =1$ or $2$, let $N(C_{\E_i})$ be the subcategory of $N(C)$ spanned by edges in $\E_i$. Then we  have a natural map \[ \op{dir}_{\E_i}: N(C_{\E_i}) \to \dd^*_2 N(C)^{\op{cart}}_{\E}  \] which on the level of $1$-simplices sends a morphism $ f: X \to Y$ to a square of the form 
		\begin{equation}
			\begin{tikzcd}
				X \arrow[r,"f"] \arrow[d,"\op{id}"] & Y \arrow[d,"\op{id}"] \\
				X \arrow[r,"f"] & Y.
			\end{tikzcd}
		\end{equation}
		We call the map $\op{dir}_{\E_i}$ as the \textit{restriction map along direction $i$}.
	\end{notation}
	
	\subsection{The enhanced operation map for $\SH(X)$.}
	In this subsection, we introduce the enhanced operation map on schemes. Let us first introduce the formal setup of the enhanced operation map as mentioned in \cite[Section 9.4]{Robalothesis}. 
\begin{notation}\label{mainnot}
		Let $\Schf$ be the category of Noetherian schemes of finite Krull dimension
		Let  \[ \D^{\otimes}: N(\Schf)^{op} \to \op{CAlg}(\op{Pr}^L_{stb}) \] be a functor. The underlying $\infty$-category of the symmetric monoidal $\infty$-category $\D^{\otimes}(X)$ is denoted by $\D(X)$. For a morphism of schemes $f:X \to Y$, we shall denote the pullback functor $\D(f): \D(Y) \to \D(X)$ by $f^*$. It is a colimit preserving functor. Thus by adjoint functor theorem, there exists a right adjoint $f_*$. We assume the functor $\D$ has the following properties:
		\begin{enumerate}
			\item  For any smooth morphism of finite type $f$, $f^*$ has a left adjoint $f_{\#}$ such that: 
			\begin{enumerate}
				\item  (Smooth projection formula) For  any $E \in \D(Y)$ and $B \in \D(X)$, the natural map formed by adjunction
				\begin{equation}\label{smothpro}
					f_{\#}(E \otimes f^*(B)) \to f_{\#}E \otimes B \end{equation} is an equivalence.
				\item  (Smooth base change) For a cartesian square of schemes 
				\begin{equation}\label{8.1}
					\begin{tikzcd}
						X' \arrow[r,"f'"] \arrow[d,"g'"] & Y' \arrow[d,"g"] \\
						X \arrow[r,"f"] & Y
					\end{tikzcd}
				\end{equation}
				with $f$ smooth of finite type, the commutative square  
				
				\begin{equation}\label{8.2}
					\begin{tikzcd}
						\D(X') & \D(Y') \arrow[l,"\lbrace f '\rbrace^*"]  \\
						\D(X ) \arrow[u,"\lbrace g'\rbrace^*"] & \D(Y) \arrow[l,"f^*"] \arrow[u,"g^*"]
					\end{tikzcd}
				\end{equation}
				is horizontally left-adjointable (\cite[Definition 4.7.4.13]{HA}), i.e.. there exists a commutative square
				\begin{equation}\label{8.3}
					\begin{tikzcd}
						\D(X') \arrow[r,"\lbrace f'\rbrace_{\#}"]& \D(Y') \\
						\D(X ) \arrow[u,"\lbrace g'\rbrace^*"] \arrow[r,"f_{\#}"] & \D(Y) \arrow[u,"g^*"]
					\end{tikzcd}
				\end{equation}
			\end{enumerate}
			\item For $f: Y \to X$ a proper morphism of schemes, $f^*$ admits a right adjoint functor $f_*$ with the following properties:
			\begin{enumerate}
				\item (Proper projection formula) For $E \in \D(Y)$ and $B \in \D(X)$, the natural map 
				\begin{equation}\label{proppro}
					f_*(E) \otimes B \to f_*(E \otimes f^*(B)) 
				\end{equation}
				is an equivalence.
				\item (Proper base change) For the cartesian square in \cref{8.1}, the induced pullback square \cref{8.2} is horizontally right adjointable. In other words, the square commutes
				\begin{equation}\label{8.4}
					\begin{tikzcd}
						\D(X') \arrow[r," f'_*"]& \D(Y') \\
						\D(X ) \arrow[u,"\lbrace g'\rbrace^*"] \arrow[r,"f_*"] & \D(Y) \arrow[u,"g^*"]
					\end{tikzcd}
				\end{equation}
			\end{enumerate} 
			\item (Support property) For a cartesian diagram of schemes in \cref{8.1} where $f$ is an open immersion and $g$ is an proper,the commutative diagram in \cref{8.3} written as  square
			\begin{equation}\label{8.5}
				\begin{tikzcd}
					\D(X) \arrow[r,"f_{\#}"] \arrow[d,"\lbrace g' \rbrace ^*"] & \D(Y) \arrow[d,"g^*"] \\
					\D(X') \arrow[r,"\lbrace f' \rbrace_{\#}"] & \D(Y')
				\end{tikzcd}
			\end{equation} 
			is horizontally right adjointable, i.e. the square 
			
			\begin{equation}\label{8.6}
				\begin{tikzcd}
					\D(X') \arrow[r,"\lbrace f' \rbrace_{\#}"] \arrow[d,"g'_*"] & \D(Y') \arrow[d,"g_*"] \\
					\D(X) \arrow[r,"f_{\#}"] & \D(Y) 
				\end{tikzcd}
			\end{equation}
			commutes.
			
		\end{enumerate}
	\end{notation}
	\begin{example}
		The functor $\Sho$ satisfies the conditions of \cref{mainnot} (\cite[Example 9.4.6]{Robalothesis}). 
	\end{example}
	\begin{notation}
		Let 
		\begin{equation}
			\begin{tikzcd}
				Y_0 \arrow[r,"u"] \arrow[d,"f_0"] & Y_1 \arrow[d,"f_1"]\\
				X_0 \arrow[r,"v"] & X_1
			\end{tikzcd}
		\end{equation}
		be an edge in $\op{Fun}(\Delta^1,\Schf)$.
		We want to denote some specific collection of edges in \linebreak $\op{Fun}(\Delta^1,N(\Schf))$ as follows:
		\begin{enumerate}
			\item $F:=$ all such squares such that $u$ and $v$ are separated morphisms of finite type.
			\item $\op{ALL}:=$ all edges in $\op{Fun}(\Delta^1,N(\Schf))$.
		\end{enumerate}
	\end{notation}
	
	\begin{theorem}\label{eomaptheorem}\cite[Section 3.2]{liu2017enhanced}
		Given a functor $\D^{\otimes}: N(\Schf)^{op} \to \op{CAlg}(\op{Pr}^L_{stb})$ which satisfies properties in \cref{mainnot}, then there exists an enhanced operation map 
		\begin{equation}\label{eomap}
			\op{EO}(\D^{\otimes}):\dd^*_{2,\{2\}}\op{Fun}(\Delta^1,N(\Schf))^{\op{cart}}_{F,\op{ALL}}  \to \op{Mod}(\op{Pr}^L_{stb}).
		\end{equation}
		which when restricted to direction $F$, gives us the lower shriek functor $\SH_!$. Morever, we get the projection and base change formulas by evaluating at specific simplices of \linebreak $\dd^*_{2,\{2\}}\op{Fun}(\Delta^1,N(\Schf))^{\op{cart}}_{F,\op{ALL}}$.
	\end{theorem}
	\begin{remark}
		\begin{enumerate}
			\item The above theorem can be applied to the functor $\D^{\otimes} = \Sho$ as $\Sho$ satisfies all conditions of \cref{mainnot}. This constructs the lower shriek functors, projection formula and base change (see \cite[Theorem 9.4.36]{Robalothesis}).
			\item The enhanced operation map takes values in $\op{Mod}(\op{Pr}^L_{stb})$ as we wanted. Let us try to explain the source of enhanced operation and justify the motivation that we gave in the beginning of this chapter. In order to encode the module objects, the source consists morphisms of schemes as objects. Hence this motivates considering the functor category $\op{Fun}(\Delta^1,\Sc)$. The functor $\dd^*_{2,\{2\}}(-)$ encodes pullback squares as $1$-simplices as we explained. Taking opposite direction along $\{2\}$ is motivated from the fact that the pullback functor is contravariant. 
			\item The construction of the enhanced operation map is technical (see \cite[Section 3.2]{liu2017enhanced} and \cite[Section 9.4]{Robalothesis}). It involves the theorem of partial adjoints (\cite[Proposition 1.4.4]{liu2017enhanced}) and $\infty$-categorical gluing for compactifiable morphisms (\cite[Theorem 0.1]{Gluerestnerv}). We shall explain how the enhanced operation map is defined on $0$ and $1$-simplices. The description of $\op{EO}(\D^{\otimes})$ on lower simplices shall help us to understand how the lower shriek functors, projection formula and base change are encoded in $\op{EO}(\D^{\otimes})$.
		\end{enumerate}
		
	\end{remark}

Let us explain the map $\op{EO}(\D^{\otimes})$ on the level of $0$ and $1$ simplices. The $0$ and $1$ simplices of $\dd^*_{2,\{2\}}(\op{Fun}(\Delta^1,N(\Schf)))_{F,\op{ALL}}$ are:
	
	\begin{enumerate}\label{eomapsource}
		\item $0$- simplices are maps of schemes $f: Y \to X$.
		\item  A morphism from $f_0:Y_0 \to X_0$ to $f_3: Y_3 \to X_3$ is a morphism of the form $\Delta^1 \times \Delta^1 \to \op{Fun}(\Delta^1,N(\Schf))$ with conditions of edges and pullback squares. Explicitly, it is a cube of the form
		\begin{equation}\label{cube}
			\begin{tikzcd}[row sep=scriptsize,column sep=scriptsize]
				Y_0  \arrow[dd,"f_0"]\arrow[rr,"v'"] \arrow[dr,"u'"] && Y_1 \arrow[dr,"u"] \arrow[dd,"f_1", yshift=1ex]  & {}\\
				& Y_2 \arrow[dd,"f_2", yshift=2ex]\arrow[rr,"v",crossing over,swap, xshift=-1ex]  && Y_3 \arrow[dd,"f_3"]  \\
				X_0 \arrow[dr,"p'",xshift=1ex] \arrow[rr,"q'",crossing over, xshift=-1ex] && X_1 \arrow[dr,"p"] & {} \\
				{} & X_2 \arrow[rr,"q",crossing over]  && X_3 
			\end{tikzcd}
		\end{equation}
		where the diagonal maps are separated morphisms of finite type and the top and bottom squares are cartesian. The horizontal squares are considered on the opposite direction.
	\end{enumerate}
	In general, $n$-simplices of the source of $\op{EO}(\D^{\otimes})$ are maps  \[ F_n: \Delta^n \times (\Delta^n)^{op}\times \Delta^1 \to \Schf \] such that $F_n|_{\Delta^n \times (\Delta^n)^{op} \times \{i\}}$ for $i =0,1$ determines an $n$-simplex in  $\delta^*_{2,\{2\}}(\Schf)_{F',\op{ALL}'}^{\op{cart}}$ where $F'$ is set of edges which are separated and finite type and $\op{ALL}'$ are all morphisms of schemes. \\
	
	Let us explain what $\op{EO}(\D^{\otimes})$ does on the level of $0$ and $1$ simplices. 
	\begin{enumerate}\label{eomaplevel}
		\item $\op{EO}(\D^{\otimes})(f):= (\D^{\otimes}(X),\D(Y))$ where $f: Y \to X$ is an arbitrary morphism of schemes. Here $\D(Y)$ is a $\D^{\otimes}(X)$-module via $f^*$.
		\item For a cube of the form \cref{cube}, $\op{EO}(\D^{\otimes})$ sends the cube to a morphism of modules
		\[ (\D(X_0)^{\otimes},\D(Y_0)) \to (\D(X_3)^{\otimes},\D(Y_3)).   \] 
		This shall help us to encode the projection formula which is a statement formulated as morphism of modules.
	\end{enumerate}

	Now we explain how the extraordinary pushforward, projection formula and base change are encoded in the map $\op{EO}(\D^{\otimes})$. 
	
	\begin{enumerate}\label{eopushprojbas}
		\item \underline{\textbf{The enhanced pullback:}}
		The map $\op{EO}(\D^{\otimes})$ encodes the map $\D^{\otimes}$. 
		When restricted the map along the direction $\op{ALL}$ (\cref{directionmapnotation}), then the induced map 
		\[ \op{EO}(\D^{\otimes})^*: \Schf^{op} \xrightarrow{X \to (X \to \Sp \mathbf{Z})} \op{Fun}(\Delta^1,\Schf)^{op} \xrightarrow{\op{EO}(\D^{\otimes})\circ \op{dir}_{\op{ALL}}}
		\op{Mod}(\op{Pr}^L_{stb})  \]
		sends a scheme $X$ to the module object $(\D^{
			\otimes}(X),\D(X))$ and sends a morphism of schemes $ f: X \to Y$ to the pullback morphism 
		\[ (\op{id},f^*): (\D^{\otimes}(\Sp \mathbf{Z}),\D(Y)) \to (\D^{\otimes}(\Sp \mathbf{Z}),\D(X)). \] 
		This is called the enhanced pullback map. Restricting it to the first coordinate, we get the map \[ \D: \Schf^{op} \to \op{Pr}^L_{stb}. \]
		\item \underline{\textbf{The extraordinary pushforward:}}
		We have a canoncial map 
		\[ \op{dir}_F:\op{Fun}(\Delta^1,\Schf') \to \dd^*_{2,\{2\}}(\op{Fun}(\Delta^1,\Schf))^{\op{cart}}_{F,\op{ALL}}   \] which is the restriction direction along $F$ (\cref{directionmapnotation}).\\ This induces the map
		\[ \op{EO}(\D^{\otimes})_!: \Schf' \to \op{Fun}(\Delta^1,\Schf') \xrightarrow{\op{EO}(\D^{\otimes})\circ\op{dir}_F} \op{Mod}(\op{Pr}^L_{stb}) \] where the first map is induced by $X \to (X \to \Sp\mathbf{Z})$.\\
		Combining with the description of $\op{EO}(\D^{\otimes})$ in \cref{eomaplevel}, the map $\op{EO}(\D^{\otimes})_!$ sends a morphism $g: Y \to X$ in $\Schf'$ to a morphism of modules
		\[(\op{id},g_!): (\D^{\otimes}(\Sp \mathbf{Z}),\D(Y))\to (\D^{\otimes}(\Sp \mathbf{Z}),\D(X)). \]
		We call $\op{EO}(\D^{\otimes})_!$ the \textit{enhanced extraordinary pushforward map}.	\\
		Via the restriction functor $\op{Mod}(\op{Pr}^L_{stb}) \to \op{CAlg}(\op{Pr}^L_{stb})$, we get the extraordinary pushforward functor 
		\[ \SH_!: \Schf' \to \op{Pr}^L_{stb} \]
		which sends a morphism $ f: X \to Y$ to the functor $f_!: \SH(X) \to \SH(Y)$.	 
		\item $\underline{\textbf{Projection formula:}}$ 
		For $f: Y \to X$ be a separated and finite type morphism, consider the cube:
		\begin{equation}
			\begin{tikzcd}[row sep=scriptsize,column sep=scriptsize]
				Y \arrow[dd,"f"] \arrow[dr,"f"] \arrow[rr,"\op{id}"] && Y \arrow[dr,"f"] \arrow[dd,"f", yshift=2ex]  & {}\\
				& X \arrow[dd,"\op{id}", yshift=-1ex] \arrow[rr,crossing over,xshift=1ex,"\op{id}"] && X \arrow[dd,"\op{id}"]  \\
				X \arrow[dr,"\op{id}"]\arrow[rr,xshift=1ex,"\op{id}", crossing over]  &&  X  \arrow[dr,"\op{id}"] & {} \\
				{} & X  \arrow[rr,"\op{id}"]  && X
			\end{tikzcd}
		\end{equation}
		
		Evaulating $\op{EO}(\D^{\otimes})$ (using the description in \cref{eomaplevel}) on this cube yields a morphism of modules 
		\[ (\op{id},f_!):(\D^{\otimes}(X),\D(Y)) \to (\D^{\otimes}(X),\D(X))  \] where $\D(Y)$ is a $\D(X)^{\otimes}$-module via $f^*$ and $\D(X)$ is a $\D(X)^{\otimes}$-module via the tensor product. This equivalent to the module homomorphism (\cref{modhomomorphism}) explained in the beginning of the chapter and hence it gives us the projection formula.

		\item $\underline{\textbf{Base change:}}$ Consider the cartesian square of schemes 
		\begin{equation}
			\begin{tikzcd}
				X' \arrow[r,"f'"] \arrow[d,"g'"] & Y' \arrow[d,"g"]\\
				X \arrow[r,"f"] & Y
			\end{tikzcd}
		\end{equation}
		where $f$ and $f'$ are separated morphism of finite type. Let us explain how the map $\op{EO}(\D^{\otimes})$ encodes the base change.\\
		The above pullback square gives us a $1$-simplex $\gamma$ in $\dd^*_{2,\{2\}}\op{Fun}(\Delta^1,\Schf)^{\op{cart}}_{F,\op{ALL}}$ which is cube:
		
		\begin{equation}
			\begin{tikzcd}[row sep=scriptsize,column sep=scriptsize]
				X' \arrow[dd] \arrow[dr,"g'"] \arrow[rr,"f'"] && Y' \arrow[dr,"g"] \arrow[dd]  & {}\\
				&X\arrow[dd]\arrow[rr,"f",,crossing over,swap, xshift=1ex]   && Y \arrow[dd] \\
				\Sp \mathbf{Z} \arrow[rr]\arrow[dr]  && \Sp\mathbf{Z}\arrow[dr] & {} \\
				{} & \Sp\mathbf{Z} \arrow[rr] &&\Sp\mathbf{Z} 
			\end{tikzcd}
		\end{equation} 
		
		The above $1$-simplex gives us two $2$-simplices $\sigma$ and $\tau$ in $$\dd^*_{2,\{2\}}\op{Fun}(\Delta^1,
		\Schf)^{\op{cart}}_{F,\op{ALL}}$$ which are morphisms  \[\Delta^2 \times (\Delta^2)^{op} \times \Delta^1 \to \Schf\]  whose upper $3 \times 3$ grids are:
		\begin{enumerate}
			\item 
			\begin{equation}\label{ss}
				\sigma|_{\Delta^2 \times \Delta^2 \times \{0\}}:=
				\begin{tikzcd}[row sep=scriptsize,column sep=scriptsize]
					Y' \arrow[d,"g"] & Y'\arrow[l]\arrow[d,"g"] & X'\arrow[l,"f'"] \arrow[d,"g'"]\\
					Y \arrow[d] & Y \arrow[l]\arrow[d] & X \arrow[l,"f"] \arrow[d]\\
					Y & Y \arrow[l] & X \arrow[l,"f"]
				\end{tikzcd}
			\end{equation}
			
			\item	
			\begin{equation}\label{tt}
				\tau|_{\Delta^2 \times \Delta^2 \times \{0\}}:=
				\begin{tikzcd}[row sep=scriptsize,column sep=scriptsize]
					Y' \arrow[d] & X'\arrow[l,"f'"]\arrow[d] & X'\arrow[l] \arrow[d]\\
					Y' \arrow[d,"g"] & X'\arrow[l,"f"]\arrow[d,"g'"] & X'\arrow[l] \arrow[d,"g'"]\\
					Y & X \arrow[l,"f"] & X \arrow[l]
				\end{tikzcd}
			\end{equation}
		\end{enumerate}
		As $d^2_1(\sigma)= d^2_1(\tau) = \gamma$, \\
		We have \[ \op{EO}(\T^{\otimes})(\gamma) \cong \op{EO}(\T^{\otimes})^*(g)\circ \op{EO}(\T^{\otimes})_!(f) = (\op{id},g^*\circ f_!)  \]
		and \[ \op{EO}(\T^{\otimes})(\gamma) \cong \op{EO}(\T^{\otimes})_!(f') \circ \op{EO}(\T^{\otimes})^*(g') = (\op{id},f'_!\circ g'^*). \]
		Restricting to the second coordinate, we have \[ f_! \circ g^* \cong g'^* \circ f'_!  \]
		which is the base change formalism. 
		
	\end{enumerate}

	\subsection{Proof of \cref{thmmain2}.}\label{proofofthmmain2}

	Let us denote the collection of squares (i.e. morphisms in $\op{Fun}(\Delta^1,\Nst)$) 
	\begin{equation}
		\begin{tikzcd}
			\X_0 \arrow[r,"f"] \arrow[d] & \X_1 \arrow[d] \\
			\Y_0 \arrow[r,"f'"] & \Y_1
		\end{tikzcd}
	\end{equation}
	where $f$ and $f'$ are separated of finite type by $F'$. The idea of the proof of \cref{thmmain2} is the following: \\
	Suppose we construct a morphism 
	\begin{equation}
		\op{EO}(\Shext): \dd^*_{2,\{2\}}\op{Fun}(\Delta^1,\Nst)_{F',\op{ALL}} \to \op{Mod}(\op{Pr}^L_{stb})
	\end{equation}
	which extends $\op{EO}(\Sho)$, then the discussion in \cref{eopushprojbas} gives us the lower shriek functor, projection formula and base change for our functor $\Shext$. Thus proving \cref{thmmain2} is reduced to extending the functor $\op{EO}(\Sho)$ to $\op{Nis-locSt}$. We extend the functor in a two step process as we did for extending $\SH$ from schemes to algebraic stacks.
	Thus we formulate a proposition in setting of category of stacks admitting $\T$-local sections (\cref{stacklocsecdef}). \\
	The proposition is a special case of the DESCENT program stated in \cite[Theorem 4.1.8]{liu2017enhanced}). Our proof is inspired from the proof of Liu and Zheng and give a new proof of the theorem. Before stating the proposition, let us fix notations in the context of category of stacks admitting $\T$-local sections.
	
	\begin{notation}
		Let $\E'$ be a collection of edges in $N^D_{\bb}(\op{St}_{\Ca})$ which are representable in $\Ca$, stable under pullback and compositions. We denote the collection of edges in $\E'$ which are in $\Ca$ by $\E$. \\
		We shall denote the collection of commutative squares 
		\begin{equation}
			\begin{tikzcd}
				\X_0 \arrow[r,"f"] \arrow[d,"v'"] & \X_1 \arrow[d,"v"] \\
				\Y_0 \arrow[r,"f'"] & \Y_1
			\end{tikzcd}
		\end{equation}
		where $f ,f' \in \E'$ by $F'$. The collection of all such pullback squares in $\Ca$ shall be denoted by $F$. \\
		We assume that there exists a functor 
		\[ \op{EO}(\D^{\otimes}) := \dd^*_{2,\{2\}}\op{Fun}(\Delta^1,N(\Ca))^{cart}_{F,\op{ALL}} \to \op{Mod}({Pr}^L_{stb}) \]
		which when restricted to the direction $\op{ALL}$ gives an $\infty$-sheaf \[ \D'^{\otimes}:\op{Fun}(\Delta^1,N(\Ca))^{op} \to \op{Mod}(\op{Pr}^L_{stb}) \]
		with respect to the topology induced by $\T$ on the functor category in a canonical way (the coverings on a object in $\op{Fun}(\Delta^1,N(C))$ are given by commutative squares in $C$ where the vertical arrows are coverings).  
	\end{notation}
	
	\begin{proposition}\label{eomapext}
		The functor $\op{EO}(\D^{\otimes})$ extends to a functor
		\begin{equation}
			\op{EO}(\D^{\otimes}_{ext}): \op{EO}(\op{St}_{\Ca}):=\dd^*_{2,\{2\}}\op{Fun}(\Delta^1,N^D_{\bb}(\op{St}_{\Ca}))_{F',\op{ALL}} \to \op{Mod}(\op{Pr}^L_{stb}).
		\end{equation}
	\end{proposition}    
	\textbf{Idea of the proof:}
	The proof is similar to the proof of \cref{thmmain}. \\
	
	We shall denote the simplicial set $
	\dd^*_{2,\{2\}}\op{Fun}(\Delta^1,\op{Cov}(\op{St}_{\Ca}))_{F',\op{ALL}}$ by $\op{EOCov}(\op{St}_{\Ca})$. We also denote the projection map \[\op{EOCov}(\op{St}_{\Ca}) \to \op{EO}(\op{St}_{\Ca}) \] induced by the map $p: \op{Cov}(\op{St_{\Ca}}) \to N^D_{\bb}(\op{St}_{\Ca})$ by $p_{EO}$.\\
	The following claim implies that the morphism $p_{EO}$ is surjective on every simplex.
	\begin{claim}
		Let $\tau_n$ be a morphism $\tau_n:\Delta^n \times \Delta^n \to N^D_{\bb}(\op{St}_{\Ca})$ which is an $n$-simplex of $\dd^*_{2,\{2\}}(N^D_{\bb}(\op{St}_{\Ca}))^{cart}_{F',\op{ALL}}$. Then there exists a map \[\tau^1_n: \Delta^1 \times \Delta^n \times \Delta^n \to N^D_{\bb}(\op{St}_{\Ca}) \] such that 
		\begin{enumerate}
			\item $\tau^1_n|_{[0] \times \Delta^n \times \Delta^n}$ is a $n$-simplex of $\dd^*_{2,\{2\}}N(\Ca)^{cart}_{F,\op{ALL}}$,
			\item $\tau^1_n|_{[1] \times \Delta^n \times \Delta^n} = \tau_n$ and
			\item $\tau^1_n: [k] \times [i] \times [j]: \to N^D_{\bb}(\op{St}_{\Ca})$ is a \v{C}ech nerve of a morphism admitting $\T$-local sections for all $0 \le i ,j \le n$ and $ 0 \le k \le 1$.
		\end{enumerate}
		
	\end{claim}
	The claim implies surjectivity of $p_{EO}$ because one can apply the claim to upper and lower pullback squares of any $n$-simplex of $\dd^*_{2,\{2\}}\op{Fun}(\Delta^1,N^D_{\bb}(\op{St}_{\Ca}))^{cart}_{F,\op{ALL}}$. By taking fiber products, this produces an edge in $\op{Fun}((\Delta^n)^{op} \times \Delta^n,\op{Fun}(\Delta^1,N^D_{\bb}(\op{St}_{\Ca})))$. Considering the \v{C}ech nerve of the edge gives us an $n$-simplex of $\op{EOCov}(\op{St}_{\Ca})$.
	\begin{proof}
		We prove it for $n=1$. The case of higher $n$ follows from induction choosing a compatible choice of atlas. We want to show that for a pullback square of the form 
		\begin{equation}
			\begin{tikzcd}
				\X_1 \arrow[d,"f_0"] & \X_3 \arrow[l,"g_1"] \arrow[d,"f_2"] \\
				\X_0 & \X_2 \arrow[l,"g_0"]
			\end{tikzcd}
		\end{equation}
		in $\op{St}_{\Ca}$ where $f_0$ and $f_2$ are in $\E'$, there exists a cube of the form 
		\begin{equation}
			\begin{tikzcd}[row sep=scriptsize,column sep=scriptsize]
				X_1  \arrow[dd,"f'_0"] \arrow[dr,"h_1"] && X_3 \arrow[dr,"h_3"] \arrow[dd,"f'_2"] \arrow[ll] & {}\\
				& \X_1 \arrow[dd,"f_0", swap]  && \X_3 \arrow[dd,"f_2"] \arrow[ll,crossing over] \\
				X_0 \arrow[dr,"h_0"]  && X_2 \arrow[dr,"h_2"]\arrow[ll,crossing over] & {} \\
				{} & \X_0  && \X_2 \arrow[ll,"g_0"]
			\end{tikzcd}
		\end{equation}
		where the square formed by vertices of $X_0,X_1,X_2$ and $X_3$ is a pullback square, $f'_0,f'_2 \in \E$ and $h_0,h_1,h_2,h_3$ are atlases admitting $\T$-local sections.\\
		
		Let $h_0: X_0 \to \X_0$ be an atlas admitting $\T$-local sections. By \cref{covnonempty}, there exists a commutative square of the form 
		\begin{equation}
			\begin{tikzcd}
				X_0 \arrow[d,"h_0"] & X_2 \arrow[l] \arrow[d,"h_2"] \\
				\X_0 & \X_2 \arrow[l]
			\end{tikzcd}
		\end{equation}
		where $h_0$ and $h_2$ are atlases admitting $\T$-local sections. As $f_0$ is representable, the base change morphism $h_1: X_1:= X_0 \times_{\X_0} \X_1 \to \X_1$ is an atlas admitting $\T$-local sections. Also $f'_0: X_1 \to X_0$ lies in $\E$. Then defining $ X_3:= X_1 \times_{X_0} X_2$ gives us the cube that we wanted. 
		
	\end{proof}
	We shall denote the collection of edges in $\op{Fun}(\Delta^1,\op{Cov}(\op{St}_{\Ca}))$ (i.e. morphisms of the form $\Delta^1 \times \Delta^1 \to \op{St}_{\Ca}$) which are of the form: 
	\begin{equation}
		\begin{tikzcd}
			(\X,x: X \to \X) \arrow[d] \arrow[r,"f"] & (\X,x': X' \to \X) \arrow[d]\\
			(\Y,y: Y \to \Y) \arrow[r,"f'"] & (\Y,y': Y'\to \Y)
		\end{tikzcd}
	\end{equation}
	by $R'$.
	We shall also denote the simplicial set $ \dd^*_{2,\{2\}}\op{Fun}(\Delta^1,\op{Cov}(\op{St}_{\Ca}))[R^{-1}]^{cart}_{F',\op{ALL}}$ by \\ $\op{EOCov}(\op{St}_{\Ca})[R'^{-1}]$. We have the following claim.
	\begin{claim}\label{catequivext}
		The map \[p'_{EO}: \op{EOCov}(\op{St}_{\Ca})[R'^{-1}] \to \op{EO} (\op{St}_{\Ca}) \] induced by the map  $p'':\op{Fun}(\Delta^1,\op{Cov}(\op{St}_{\Ca}))[R^{-1}] \to \op{Fun}(\Delta^1, N^D_{\bb}(\op{St}_{\Ca}))$ is categorical equivalence of simplicial sets.
	\end{claim}
	\begin{proof}
		Following the arguments in the proof of \cref{thmmain}, we see that the map $p''$ is a categorical equivalence. Thus it admits a categorical inverse $q''$. It can be easily verified that $q''$ sends pullback squares to pullback squares. Thus the map $q''$ induces a map \[q'_{EO}: \op{EO} (\op{St}_{\Ca})  \to  \op{EOCov}(\op{St}_{\Ca})[R'^{-1}].\] 
		As $q'' \circ p'' = \op{id}_{\op{Fun}(\Delta^1,\op{St}_{\Ca})}$, we have $p'_{EO}\circ q'_{EO} = \op{id}_{\op{EO}(\op{St}_{\Ca})}$. On the other hand, we see that $q'_{EO} \circ p'_{EO}$ is preisomorphic to $\op{id}_{\op{EOCov}(\op{St}_{\Ca})[R'^{-1}]}$ in the sense of \cite[23.4]{Rezk}. Thus by \cite[Lemma 23.8]{Rezk}, we get that $p'_{EO}$ is a categorical equivalence. 
	\end{proof}
	We also prove another claim regarding the monomorphism \[i_{EO}: \op{EOCov}(\op{St}_{\Ca}) \hookrightarrow \op{EOCov}(\op{St}_{\Ca})[R'^{-1}]. \]
	\begin{claim}\label{localinthmmain2}
		Let  $\E$ be an $\infty$-category and let $F: \op{EOCov}(\op{St}_{\Ca}) \to \E$ be a functor which maps $R'$ to equivalences. Then $F$ extends to a functor $F': \op{EOCov}(\op{St}_{\Ca})[R'^{-1}] \to \E$. 
	\end{claim}
	\begin{proof}
		We construct the functor $F'$ inductively. As objects of $\op{EOCov}(\op{St}_{\Ca})[R'^{-1}]$ are same as objects of $\op{EOCov}(\op{St}_{\Ca})$, we define $F'$ as $F$. \\
		Assume we have defined $F'$ upto $n-1$ simplices. Let $\sigma_n $ be an $n$-simplex of $\op{EOCov}(\op{St}_{\Ca})$. By induction, the boundary of $\sigma_n$ maps to a morphism $ F'(\partial(\sigma_n)): \partial\Delta^n \to \D$. \\
		We denote the image of $\sigma_n$ via $p'_{EO}$ by $\sigma'_n$. As $p_{EO}$ is surjective on each simplex, this lifts to an element $\sigma''_n: \Delta^n \to \op{EOCov}(\op{St}_{\Ca})$. Following the arguments of \cref{covislocalization}, we get a morphism  \[ \tau_n: \partial\Delta^n \times \Delta^1 \coprod_{\partial\Delta^n \times\{0\}} \Delta^n \to \D \]
		where 
		\begin{enumerate}
			\item $\tau_n|_{[0] \times \Delta^n} = F(\sigma''_n)$.
			\item $\tau_n|_{[1] \times \partial\Delta^n} = F'(\partial(\sigma_n)) $, and
			\item $\tau_n|_{\Delta^1 \times [k]} $ is an equivalence for all $ 0 \le k \le n$. 
		\end{enumerate}
		As explained in \cref{covislocalization}, this morphism extends to a morphism  \[ \tau'_n: \Delta^n \times \Delta^1 \to 
		\D. \] The morphism $\tau'_n$ when restricted to $\Delta^n \times [1]$ gives us the morphism 
		\[ F'(\sigma_n):\Delta^n \to \D. \] This completes the proof of induction and hence the claim.
	\end{proof}
	\begin{proof}[Proof of \cref{thmmain2}]
		We define a morphism 
		\begin{equation}
			\varphi: \op{EOCov}(\op{St}_{\Ca})
			\xrightarrow{\op{EO}(\D^{\otimes})} \op{Fun}(N(\Delta),\op{Mod}(\op{Pr}^L_{stb}))  \xrightarrow{\op{res}_{[-1]} \circ i} \op{Mod}(\op{Pr}^L_{stb})
		\end{equation}
		as follows:
		\begin{enumerate}
			\item Note that we have a canonical morphism 
			\[ \op{EOCov}(\op{St}_{\Ca}) \to \op{Fun}(N(\Delta),\op{EO}(\Ca)). \]
			The morphism $\op{EO}(\D^{\otimes})$ is the functor $\op{EO}(\D^{\otimes})$ applied to $\op{Fun}(N(\Delta),\op{EO}(\Ca))$.  
			\item The maps $i$ and $\op{res}_{[-1] \times \Delta^n}$  are the same maps that we defined in the construction of $\phi$ in \cref{thmmain}. The functorial association of limit is possible as $\op{Mod}(\op{Pr}^L_{stb})$ admits small limits (\cite[Proposition 3.2.2.1]{HA}). 
		\end{enumerate}
		
		Similar to the arguments in \cref{thmmain}, we see that the morphism $\varphi$ sends $R$ to equivalences. By \cref{localinthmmain2}, this induces a map \[\varphi': \op{EOCov}(\op{St}_{\Ca})[R^{-1}] \to \op{Mod}(\op{Pr}^L_{stb}). \]
		By \cref{catequivext}, this induces a morphism  
		\[ \op{EO}(\D^{\otimes}_{ext}): \op{EO}(\op{St}_{\Ca})  \to \op{Mod}(\op{Pr}^L_{stb}). \]
		It is automatically clear that this an extension of the morphism $\op{EO}(\D^{\otimes})$. This completes the proof.
	\end{proof}
	
	\begin{remark}
		The functor $EO(\D^{\otimes}_{ext})$ when restricted to direction ALL is indeed the functor $\D'^{\otimes}_{ext}$ obtained by applying \cref{thmmain} to functor $D'^{\otimes}$). 
	\end{remark}

	To conclude the proof of \cref{thmmain2}, we apply \cref{eomapext} to the functor $\op{EO}(\Sho)$ and $F'$ to be collection of morphisms which are representable, separated and finite type. We first extend the functor to the category of algebraic spaces. Then, we apply it again to extend it to the $(2,1)$-category $\op{Nis-locSt}$.\\
	
	Thus this gives us the enhanced operation map 
	\begin{equation}
		\op{EO}(\Shext): \dd^*_{2,\{2\}}\op{Fun}(\Delta^1,N^D_{\bb}(\op{Nis-locSt}))^{cart}_{F',\op{ALL}} \to \op{Mod}(\op{Pr}^L_{stb}).
	\end{equation} 
	
	The composition 
	\begin{equation}
		\Nst^{op} \xrightarrow{\op{dir}_{ALL} \circ(\X \to (\X \to \Sp \mathbf{Z}))} \dd^*_{2,\{2\}}\op{Fun}(\Delta^1,N^D_{\bb}(\op{Nis-locSt}))^{cart}_{F',\op{ALL}} \xrightarrow{\op{EO}(\Shext)} \op{Mod}(\op{Pr}^L_{stb})
	\end{equation}
	sends $\X$ to $(\Shext(\mathbf{Z}), \SHe(\X))$ and it sends morphisms $f: \X \to \Y$ to the pair of morphisms $ (\op{id},f^*): (\Shext(\mathbf{Z}),\SHe(\Y)) \to (\Shext(\mathbf{Z}),\SHe(\X))$. \\
	
	As explained in \cref{eopushprojbas}, the functor $\op{EO}(\Shext)$ induced by restricting along direction $F'$ gives us the functor 
	
	\begin{equation}
		\Sheshr: N^D_{\bb}(\op{Nis-locSt}') \to \op{Pr}^L_{stb}.
	\end{equation} 
	which extends $\SH_!$.\\
	
	The projection and base change formulas also follow from $\op{EO}(\Shext)$ as explained in  beginning of \cref{proofofthmmain2} and \cref{eopushprojbas}. This completes the proof of \cref{thmmain2}. \\
	
		\section{Six operations for $\Shext(\X)$}
	
	In the previous two sections, we have extended the motivic homotopy functor from schemes to algebraic stacks and constructed the six functors. In this section, we prove other relations of six operations: homotopy invariance, localization and purity. \\
	
	In the first subsection, we state results of smooth and proper base change theorems in our context. In the second subsection, we prove the theorems of localization and homotopy invariance. In the third subsection, we construct the natural transformation $\alpha_f$. In the fourth subsection, we construction the purity transformation $\rho_f$. In the last subsection, we summarize all the results and state in a single theorem.

	\subsection{Smooth and proper base change.}
	
	We prove smooth and proper base change theorems. \cref{thmmain2} constructs the lower shriek functors for representable morphisms and separated of finite type, in particular for open immersions and proper morphisms. On the level of schemes, for a smooth morphism $f:X \to Y$, the pullback morphism $f^*: \SH(Y) \to \SH(X)$ admits a left adjoint $f_{\#}$. It is natural to expect such a result in the context of $\Shext$ of $\op{Nis-locSt}$. 
	
	\begin{lemma}\label{pullbackisconservative}
		Let $\X \in \op{Nis-locSt}$ and $x: X \to \X$ be a smooth atlas which admits Nisnevich-local sections. Then the pullback map $x^*: \SHe(\X) \to \SHe(X)$ is conservative. 
	\end{lemma}
	\begin{proof}
		This is a consequence of \cite[Proposition 4.7.5.1]{HA} applied to the $\J = N(\Delta_+)$ and $ q: N(\Delta) \to \widehat{\op{Cat}_{\infty}}$ which is the functor $\Shext$ applied to simplicial object $X^{\bb}_{x}$. We get that the functor $G:=x^*: \SHe(\X) \to \SHe(X)$ is conservative.
	\end{proof}
	
	\begin{proposition}(Smooth base change)\label{smoothbasechangeext}
		Let $ f: \Y \to \X$ be a  representable smooth morphism in $\op{Nis-locSt}$. Then $f^*$ admits a left adjoint $f_{\#}$. Morever for a cartesian square in $\op{Nis-locSt}$ of the form 
		\begin{equation}
			\begin{tikzcd}
				\X' \arrow[r,"f'"] \arrow[d,"g'"] & \Y' \arrow[d,"g"] \\
				\X \arrow[r,"f"] & \Y
			\end{tikzcd}
		\end{equation} 
		where $f$ (and thus $f'$) is smooth, we have an equivalence 
		\[ \op{Ex}(\Delta^*_{\#}):f^*g_{\#} \cong g'_{\#}f'^*.  \]
	\end{proposition}
	
	\begin{proof}
		The proof of the proposition uses the theory of left and right adjointable squares and \ref{adjsquarethm}. Let $x: X \to \X$ be an atlas admitting Nisnevich-local sections. Then $y: X \times_{\X} \Y \to \Y$ is an atlas admitting $\T$-local sections. Taking \v{C}ech nerves of $x$ and $y$, produces a morphism 
		\[ H: N(\Delta)^{op} \times \Delta^1 \to N^D_{\bb}(\op{St}_{\Ca}) \] where $H|_{[0] \times \Delta^1}= f'$ (where $f'$ is base change of $f$ along $x$), $H|_{N(\Delta_+)^{op} \times [0]} = X^{\bb}_{x}$ and $H|_{N(\Delta_+)^{op} \times [1]} = Y^{\bb}_{y}$. \\
		Composing with the functor $\SHe: N^D_{\bb}(\op{St}_{\Ca})^{op} \to \op{Pr}^L_{stb} \hookrightarrow \widehat{\op{Cat}_{\infty}}$, we get a functor 
		\[ H: N(\Delta) \to \op{Fun}(\Delta^1,\widehat{\op{Cat}_{\infty}}) \]
		As $f_{\#}$ is left adjoint of $f^*$ on the level of schemes and we have smooth base change (see \cite[Example 9.4.8]{Robalothesis}), this implies that the functor $H$ can be realized as functor:
		\[ H: N(\Delta) \to \op{Fun}^{LAd}(\Delta^1,\widehat{\op{Cat}_{\infty}}) \] where $Fun^{LAd}(\Delta^1,\widehat{\op{Cat}_{\infty}})$ is the $\infty$-category of left adjointable functors (\cite[Definition 4.7.4.13]{HA}). As $\op{Fun}^{LAd}(\Delta^1,\widehat{\op{Cat}_{\infty}})$ admits small limits (\cite[Corollary 4.7.4.18]{HA}), the map $H$ admits a limit \[\overline{H}: N(\Delta_+)^{op} \to \op{Fun}^{LAd}(\Delta^1,\widehat{\op{Cat}_{\infty}}). \] Evaluating $H'$ at $[-1]$, we get the morphism \[ f^*: \SHe(\Y) \to \SHe(\X) \] which is an element of $\op{Fun}^{LAd}(\Delta^1,\widehat{\op{Cat}_{\infty}})$. By definition of  $\infty$-category of left adjointable functors, we get that $f^*$ admits a left adjoint \[ f_{\#}: \SHe(\X) \to \SHe(\Y). \]
		
		It remains to prove the smooth base change. Let us denote the cartesian square in the proposition as a morphism $\sigma: \Delta^1 \times \Delta^1 \to N^D_{\bb}(\op{St}_{\Ca})$ Let $y: Y \to \Y$ be an atlas admitting $\T$-local sections. By \cref{covnonempty} and the fact that pullback of  representable smooth morphism is smooth, we get a morphism 
		\[ G': N(\Delta_+)^{op} \times \Delta^1 \times \Delta^1 \to \Ca \] such that
		\begin{enumerate}
			\item $ G'|_{[-1] \times \Delta^1 \times \Delta^1} = \sigma$.
			\item $G'_{N(\Delta_+)^{op} \times [k] \times [j]}$ is \v{C}ech nerve of an atlas of $\sigma([k],[j])$ admitting Nisnevich-local sections for all $0 \le j,k \le 1$. 
		\end{enumerate}
		Composing with $\SH(-)$, we get a functor 
		\[G: N(\Delta) \times \Delta^1 \times \Delta^1  \to \widehat{\op{Cat}_{\infty}}.\]
		For every $\tau: \Delta^1 \to N(\Delta) \times \Delta^1$, the induced square \[ G \circ(\tau \times \Delta^1): \Delta^1 \times \Delta^1 \to \widehat{\op{Cat}_{\infty}} \] is left adjointable by smooth base change theorem on the level of schemes (\cite[Example 9.4.8]{Robalothesis}). Applying \cite[Lemma 4.3.7]{liu2017enhanced} to the functor $G$, we get that the square 
		\begin{equation}
			\begin{tikzcd}
				\SHe(\Y) \arrow[r,"f^*"] \arrow[d,"g^*"] & \SHe(\X) \arrow[d,"g'^*"] \\
				\SHe(\Y') \arrow[r,"f'^*"] & \SHe(\X')
			\end{tikzcd}
		\end{equation}
		is the left adjointable, i.e.. we have 
		\[ \op{Ex}(\Delta^*_{\#}): g'_{\#}f'^* \cong f^*g_{\#}. \]
	\end{proof}
	\begin{remark}\label{lowershriekforopenproper}
		The lower shriek functor on the level of schemes agrees with $(-)_{\#}$ for open immersions and $(-)_*$ along proper morphisms (\cite[Theorem 9.4.8]{Robalothesis}). As the lower shriek functor on the level of algebraic stacks is constructed by taking limit along \v{C}ech covers of atlases, we get that for an open immersion $j: \X' \to \X$, we have an equivalence $j_! \cong j_{\#}$ and for  representable proper morphisms $p: \X'' \to \X$, we have an equivalence $p_! \cong p_*$.
	\end{remark}
	A similar proposition holds in the case of representable proper morphisms.
	\begin{proposition}[Proper base change]
		Given a cartesian square in $\op{Nis-locSt}$ of the form 
		\begin{equation}
			\begin{tikzcd}
				\X' \arrow[r,"f'"] \arrow[d,"g'"] & \Y' \arrow[d,"g"] \\
				\X \arrow[r,"f"] & \Y
			\end{tikzcd}
		\end{equation} 
		where $g$ (and thus $g'$) is representable and proper, we have an equivalence 
		\[ \op{Ex}(\Delta^*_*):f^*g_* \cong g'_*f'^*.  \]
	\end{proposition}
	\begin{proof}
		As $f_* \cong f_!$ when $f$ is representable and proper, the proper base change is indeed the base change with respect to lower shriek functor. This holds due to \cref{thmmain2}.
	\end{proof}
	
	\subsection{Localization and homotopy invariance.}
	
\begin{proposition}(Localization)\label{localizationext} If $i: \Z \to \X$ is a closed immersion with complementary open immersion $j: U:=X-Z \hookrightarrow \X$, we have the cofiber sequences:
		\begin{enumerate}
			\item \begin{equation}
				j_!j^! \to \op{id} \to i_*i^* 
			\end{equation}
			\item \begin{equation}
				i_!i^! \to  \op{id} \to j_*j^*
			\end{equation}
		\end{enumerate}
	\end{proposition}
	\begin{proof}
		We prove the existence of the first cofiber sequence. The second cofiber sequence is dual to the first one. Let $x: X \to \X$ be an atlas admitting Nisnevich-local sections. \\
		Then the restriction of $x$ to $\Z$ and $\U$ defines atlases $z:Z \to \Z$ and $u:U \to \U$ and these induce morphism of the \v{C}ech nerves $Z^{+}_{\bb,z} \to X^{+}_{\bb,x}$ and $U^{+}_{\bb,u} \to X^{+}_{\bb,x}$. Thus we have morphisms $Z^n_{\Z} \hookrightarrow X^n_{\X}$ and $U^n_{\U}  \hookrightarrow X^n_{\X}$ which are closed and open immersions respectively for every $n$.\\
		Let $E \in \SHe(\X)$, then $E=(E_n)_{n \in \Delta}$ where $E_n \in \SH(X^n_{\X})$.
		
		For any $n$, we have a square which is a fiber sequence 
		\begin{equation}
			\begin{tikzcd}
				j_{n\#}j_n^*(E_n) \arrow[r]\arrow[d] & E_n \arrow[d] \\
				0 \arrow[r] & i_{n*}i_n^*(E_n)
			\end{tikzcd}
		\end{equation} 
		in $\SH(X^n_{\X})$, because localization holds for schemes (\cite[Theorem 9.4.25]{Robalothesis}). \\
		
		This can be visualized as a limit map
		\[ \overline{H_n}: {\Lambda^2_0}^{\triangleleft} \to \op{Cat}_{\infty}\]
		where \[ H_n: \Lambda^2_0 \to \widehat{\op{Cat}_{\infty}} \]
		Also, we have a morphism $H_{-1}: \Lambda^2_0 \to \widehat{\op{Cat}_{\infty}}$ given by the diagram 
		\begin{equation}
			\begin{tikzcd}
				{} & E \arrow[d]\\
				0 \arrow[r] & i_*i^*E
			\end{tikzcd}
		\end{equation}
		The collection of maps $H_n$ for every $n \ge -1$ induces a morphism
		\[ H: \Lambda^2_0 \to \op{Fun}(N(\Delta_{s+}),\widehat{\op{Cat}_{\infty}}) \]
		where for every $[n] \in \Delta_{s+}$, $H|_{[n]}:= H_n$. 
		As $\op{Fun}(N(\Delta_{s+}),\widehat{\op{Cat}_{\infty}})$ admits all limits, there exists an extension of $H$ to \[ \overline{H}:\Delta^1 \times \Delta^1 \cong \Lambda^{2\triangleleft}_0  \to \op{Fun}(N(\Delta_{s+}),\widehat{\op{Cat}_{\infty}}). \]
		The morphism $\overline{H}$ is evaluated at $[-1] \in \Delta_{s+}$ gives us a square  \[ \overline{H}|_{[-1]}: \Delta^1 \times \Delta^1 \to \widehat{\op{Cat}_{\infty}}. \]
		As smooth pullbacks commute with $(-)_{\#},(-)_*$, we have $j_{\#}j^*(E) \cong (j_{n\#}j_n^*(E_n))_{n \in \Delta}$ and $i_{*}i^*(E) \cong (i_{n*}I_n^*(E_n))_{n \in \Delta}$. Thus the morphism $H'|_{[-1]}$ is the pullback square 
		\begin{equation}
			\begin{tikzcd}
				j_{\#}j^*(E) \arrow[r] \arrow[d] & E \arrow[d] \\
				0 \arrow[r] & i_*i^*(E).
			\end{tikzcd}
		\end{equation}
		This proves the localization theorem.\\
		The same argument works for the dual sequence $2$.
	\end{proof}
	\begin{proposition}\label{hominvext}(Homotopy invariance) For any stack $\X \in \op{Nis-locSt}$, the projection $\pi: \A^1_{\X} \to \X$ induces a fully faithful functor $\pi^*(-): \SHe(\X) \to \SHe(\A^1_{\X})$
	\end{proposition}
	
	\begin{proof}
		To show that $\pi_{\X}^*$ is fully faithful, we need to show that the unit transformation $u:\pi_{\#}\pi^* \to \op{id}$ is an equivalence. As $\pi_{\#}$ satisfies projection formula, we are reduced to showing $u(1_{\X}):\pi_{\#}\pi^*(1_{\X})\to 1_{\X}$ is an equivalence.\\
		
		Fixing the usual atlas $ x: X \to \X$, let $\pi_0: \A^1_X \to X$ be the projection map. As $\SH$ satisfies homotopy invariance on the level of schemes, we have an equivalence $u_0(1_X):\pi_{0\#}\pi_0^*(1_X) \to 1_X$. As $(-)_{\#}$ commutes with pullbacks, we get that pullback of $u(1_{\X})$ along $x^*$ is $u_0(1_X)$. As $x^*$ is conservative (\cref{pullbackisconservative}) and $u_0(1_X)$ is an equivalence, we get that $u(1_{\X})$ is an equivalence.
		
	\end{proof}
	
	\subsection{The natural transformation $\alpha_f$.}
	
	We construct the natural transformation $\alpha_f$ which is the extension of the natural transformation of the same notation on the level of schemes (\cite[Proposition 2.2.10]{Cisinski_2019}). We construct the natural transformation for a specific class of morphisms in $\op{Nis-locSt}$. 
	\begin{definition}
		A  representable morphism $ f: \X \to \Y$ in $\op{Nis-locSt}$ is \textit{compactifiable} if admits a factorization of the form $ \X \xrightarrow{j}\overline{\X}\xrightarrow{p}\Y$ where $j$ is an open immersion and $p$ is a proper representable morphism of algebraic stacks. 
	\end{definition}
	\begin{example}
		Open immersions and representable proper morphisms are compactifiable.\\
	\end{example}

	\begin{proposition}\label{alphaext}
		Let $f:\X \to \Y$ be a compactifiable morphism of algebraic stacks in $\op{Nis-locSt}$. Then there exists a natural transformation:
		\begin{equation} 
			\alpha_f: f_! \to f_* 
		\end{equation}
		which is an equivalence if $f$ is proper.
	\end{proposition}
	
	\begin{proof}
		The construction of $\alpha_f$ is similar to the construction on the level of schemes (\cite[Proposition 2.2.10]{Cisinski_2019}). Consider a factorization $(j,p)$ of the $f$. \\
		At first, we have a natural transformation \[j_{\#} \to j_*\] for any open immersion $j$. This follows because of the natural transfromation $\op{Ex}(\Delta_{*\#})$ (the dual of the smooth base change $\op{Ex}(\Delta^*_{\#})$) applied to the cartesian square
		\begin{equation}
			\begin{tikzcd}
				\X \arrow[d,"\op{id}"] \arrow[r,"\op{id}"] & \X \arrow[d,"j"] \\
				\X \arrow[r,"j"] & \overline{\X}.
			\end{tikzcd}
		\end{equation}
     Thus the natural transformation $\alpha_f$ is defined as 
		\[ \alpha_f: f_!=p_*j_{\#} \to p_*j_* \cong f_*.\]
		When $f$ is proper, the definition of compactifiable implies that $\alpha_f$ is an equivalence.
		
	\end{proof}

	\subsection{Homotopy purity.}
	We now prove the homotopy purity theorem of $\Shext(-)$. At first, we construct the natural transformation $\rho_f$ which is analog to the purity transformation on the level of schemes. We then prove the homotopy purity theorem using the deformation to the normal cone. Also as a corollary, we get an explicit description of the self equivalence $\op{Tw}_f$. \\
	\begin{proposition}(Purity)\label{purityext} Let $f: \X \to \Y$ to be smooth morphism separated of finite type, there exists a self equivalence $\op{Tw}_f: \SHe(\X) \to \SHe(\X)$ and an equivalence 
		\[ \op{Tw}_f \circ f^! \cong f^*. \]
	\end{proposition}
	
	\begin{proof}
		Let $ f: \X \to \Y$ be a smooth morphism separated of finite type. Let $ \delta: \X \to \X \times_{\Y} \X$ be the diagonal morphism and $ p: \X \times_{\Y} \X \to \X$ be the projection map. Then we denote the \textit{Thom transformation} \[ \Sigma_f:= p_{\#} \circ \delta_*. \]
		By the base change theorem with respect to $f_!$ (\cref{thmmain2}) and the fact that $\alpha_f$ is an equivalence for proper morphisms (\cref{alphaext}), we can construct the transformation as one does on the level of schemes (\cite[Section 2.4.20]{Cisinski_2019})
		\[ \rho_f: f_{\#} \to f_! \circ \Sigma_f.  \]
		Let us check that $\Sigma_f$ and $\rho_f$ are equivalences. When base changed to the level of schemes by choosing an atlas, these natural transformations are equivalences (\cite[Theorem 9.4.37]{Robalothesis}). As the pullback functor along an atlas is conservative (\cref{pullbackisconservative}), this implies that these natural transformations are equivalences.\\
	\end{proof}
	
	The above proposition gives us the relation between $f^*$ and $f^!$ via the self equivalence $\op{Tw}_f:= (\Sigma_f)^{-1}$. We want to have a precise description of the self equivalence $\Sigma_f$ as one gets on the level of schemes via deformation to normal cone. We prove a similar result in the context of $\Shext(-)$.\\
	At first we define the notion of smooth closed pairs in this context as one does on the level of schemes.
	
	\begin{definition}
		Let $\Y$ be an algebraic stack in $\Nst$. A \textit{smooth closed pair} over $\Y$ is a pair $(\X,\Z)$ where: 
		\begin{enumerate}
			\item $\X,\Z$ are stacks over $\Y$ in $\Nst$ such that the projection maps to $\Y$ are smooth.
			\item $\Z \hookrightarrow \X$ is a closed substack of $\X$.
		\end{enumerate}
		A morphism of smooth closed pairs $(\X,\Z) \to (\X',\Z')$ is a representable morphism of algebraic stacks $ f: \X \to \X'$ such that $f^{-1}(\Z') =\Z$ as a set. 
	\end{definition}
	
	\begin{notation}\label{puritynot}
		\begin{enumerate}
			\item For a smooth closed pair $(\X,\Z)$ over $\Y$, we denote  \begin{equation}
				\frac{\X}{\X-\Z}:= p_{\#}(i_{\Z*}(1_{\X})) \in \Shext(\Y) 
			\end{equation}
			where $ p: \X \to \Y$. \\
			
			\item Let $p:\V \to \Y$ be a vector bundle over the algebraic stack $\Y$ where $\V,\Y \in \Nst$. Let $s: \Y \to \V$ be the zero section. Then the pair is $(\V,\Y)$ a smooth closed pair where $\Y$ is realized as a closed substack of $\V$ via the zero section. 
		\end{enumerate}
	\end{notation}
	\begin{remark}
		A morphism of smooth closed pairs $ (\X,\Z) \to (\X',\Z')$ induces a map \[ \frac{\X}{\X-\Z} \to \frac{\X'}{\X'-\Z'} ~~~(\text{see}~\cite[2.4.32]{Cisinski_2019}). \]
	\end{remark}
	We now state and prove the homotopy purity theorem. Recall that for a closed substack $\Z \hookrightarrow \X$, we denote the normal cone by $N_{\Z}(\X)$ and 
	the deformation to the normal cone by $D_{\Z}(\X)$ (see \cref{blowupdeforstacknotation}).
	\begin{proposition}\label{hompurityext}
		Let $(\X,\Z)$ be a smooth closed pair over $\Y$.
		Then the canonical morphisms of smooth closed pairs 
		\begin{equation}
			(\X,\Z) \xleftarrow{p_0} (D_{\Z}(\X),\A^1_{\Z}) \xrightarrow{p_1} (N_{\Z}\X,\Z) 
		\end{equation}
		induces an equivalence 
		\begin{equation}
			\frac{\X}{\X-\Z} \cong \frac{D_{\Z}\X}{D_{\Z}\X - \A^1_{\Z}} \cong \frac{N_{\Z}\X}{N_{\Z}\X-\Z} 
		\end{equation}
		
	\end{proposition}
	
	\begin{proof}
		By \cref{blownordefnisloc}, we know that the algebraic stacks $D_{\Z}(\X)$ and $N_{\Z}(\X)$ are in $\op{Nis-locSt}$ when $\X \in \op{Nis-locSt}$.\\
		The morphisms of smooth of closed pairs in $\Nst$ 
		\begin{equation}
			(\X,\Z) \xleftarrow{p_{0}} (D_{\Z}(\X),\A^1_{\Z}) \xrightarrow{p_{1}} (N_{\Z}\X,\Z) 
		\end{equation}
		on base changed to an atlas $y: Y \to \Y$ yields us morphisms 
		\begin{equation}
			(X,Z) \xleftarrow{p_{0,0}} (D_Z(X),\A^1_Z) \xrightarrow{p_{1,0}} (N_Z(X),Z_\Z) 
		\end{equation}
		By \cref{dfm}, we have an equivalence:
		\begin{equation}
			\frac{X}{X-Z} \xleftarrow[\cong]{p_{{0,0}*}} \frac{D_Z(X)_{\X}}{D_{Z}(X)-\A^1_Z} \xrightarrow[\cong]{p_{{1,0}*}}
			\frac{N_{Z}(X)}{N_{Z}(X)-Z} 
		\end{equation}
		The construction of morphisms $p_{0*}$ commutes with pullbacks. Thus we have $y*^*(p_{0*}) \cong p_{{0,0}*}$ and $y^*(p_{1*}) \cong p_{{0,1}*}$. By \cite[Theorem 9.4.34]{Robalothesis}, we see that $p_{{0,0}*}$ and $p_{{1,0}*}$ are equivalences. As $y^*$ is conservative (\cref{pullbackisconservative}), we get that $p_{0*}$ and $p_{1*}$ are equivalences.	
	\end{proof}
	
	The above equivalence gives us an explicit description of the Thom transformation $\Sigma_f$ in terms of the Thom space of the normal bundle of $f$. 
	\begin{notation}
		Let $f: \X \to \Y$ be a smooth representable morphism of algebraic stacks. Then the Thom space of the normal bundle is  defined as \[  \op{Th}(N_f):= \frac{N_{\X}(\X \times_{\Y} \X)} {{N_{\X}(\X \times_{\Y} \X) -\Y}}.  \]
		This definition is analog to the one defined on the level of schemes (\cite[Definition 9.4.27]{Robalothesis}).
	\end{notation}

	\begin{corollary}
		Let $f: \X \to \Y$ be separated of finite type, smooth representable morphism of algebraic stacks. Consider the commutative diagram of algebraic stacks 
		\begin{equation}
			\begin{tikzcd}
				\X \arrow[dr,"\delta"] & {} & {}\\
				{} & \X \times_{\Y} \X \arrow[r,"p"] \arrow[d,"q"] & \X \arrow[d,"f"]\\
				{} & \X \arrow[r,"f"] & \Y
			\end{tikzcd}
		\end{equation}
		Then we have:
		\begin{equation}
			\Sigma_f(-):= p_{\#}\delta_*(-) \cong -\otimes \op{Th}(N_f).
		\end{equation}

	\end{corollary}
	\begin{proof}
		The proof is exactly is same as one does for $\SH$ (\cite[Eq. 9.4.88]{Robalothesis}), we have 
		\begin{align*}
			\Sigma_f(-) & = p_{\#}\delta_*(-)\\
			& =p_{\#}\delta_*((1_{\Y} \otimes \delta^*p^*(-)) \\
			& \cong p_{\#}(\delta_*(1_{\Y}) \otimes p^*(-)) ~~ (\text{projection formula})\\
			& \cong p_{\#}\delta_*(1_{\Y})\otimes(-) (\text{projection formula}) \\
			& \cong \op{Th}(N_f) \otimes (-) (\cref{hompurityext}) \\
		\end{align*}
	\end{proof}
	
	\begin{remark}
		It is possible to define the Tate object $1_{\X}(1) \in \Shext(\X)$ as $\Omega^2(K)$ where $K$ is defined as cofiber of the map $ 1_{\Y} \to p_{\#}p^*1_{\X}$ and $\Omega$ is the inverse of the suspension functor on the stable $\infty$-category $\Shext(\X)$. Then for a smooth representable morphism $f: \Y \to \X$, we have  \[ \op{Th}(N_f) \cong 1_{\Y}(d)[2d] \]
		where $d$ is the relative dimension of the morphism $f$ (and $1_X(d)$ is the $d$th iterated tensor product of $1_X(1)$). This follows from the fact that such a description holds on the level of schemes. Thus combining the description of the Thom space of normal bundle of $f$ with the purity isomorphism, we get that: 
		
		\[ f_{\#}(-) \cong f_!(1_{\X}(d)[2d] \otimes (-)). \]
		
	\end{remark}
	
	\subsection{Summarizing the results.}
	
	In this subsection, we summarize the results that we have proved in the previous sections in a single theorem.  
	
	\begin{theorem}\label{thmsixopext}
		The stable homotopy functor $\Sho: N(\Schf)^{op} \to \op{CAlg}(\op{Pr}^L_{stb})$ extends to a functor \begin{equation}
			\Shext: \Nst^{op} \to \op{CAlg}(\op{Pr}^L_{stb})
		\end{equation}
		such that for a morphism $f: \Y \to \X$ we have the following functors:
		\begin{enumerate}
			\item $f^*:  \Shext(\X) \to \Shext(\Y)$.
			\item $f_*: \SHe(\Y) \to \SHe(\X)$.
			\item $f_!: \SHe(\X) \to \SHe(\Y)$ when $f$ is representable, separated and of finite type.
			\item $f^!: \SHe(\Y) \to \SHe(\X)$ when $f$ is representable, separated and of finite type.
			\item $ - \otimes -: \SHe(\X) \times \SHe(\X)  \to \SHe(\X)$.
			\item  $ \op{Hom}_{\SHe(\X)}(-,-): \SHe(\X) \times \SHe(\X) \to \Shext(\X)$.
			
		\end{enumerate}
		The functor $\Shext$ along with the functors $(f^*,f_*,f_!,f^!,-\otimes-,\op{Hom}(-,-))$ satisfy the following properties:
		\begin{enumerate}
			\item (Monoidality) $f^*$ is monoidal, i.e. there exists an equivalence
			\begin{equation}
				f^*(E \otimes E') \cong f^*(E)\otimes f^*(E')
			\end{equation}
			for $E,E'\in \Shext(\X)$
			\item (\textit{Projection Formula}) For  $ E,E' \in \SHe(\X)$ and $ B \in \SHe(\Y)$, we have the following equivalences:
			\begin{enumerate}
				\item  \begin{equation}
					f_!(B \otimes f^*(E)) \cong (f_!B \otimes E)
				\end{equation}
				\item \begin{equation}
					f^!\op{Hom}_{\SHe(X)}(E,E') \cong \op{Hom}_{\SHe(Y)}(f^*E,f^!E')
				\end{equation}
			\end{enumerate}
			\item (\textit{Base Change}) If 
			\begin{equation}
				\begin{tikzcd}
					\X' \arrow[r,"f'"] \arrow[d,"g'"] & \Y' \arrow[d,"g"]\\
					\X \arrow[r,"f"] & \Y
				\end{tikzcd}
			\end{equation}
			is a cartesian square of base schemes with $g$ being representable, separated and of finite type, we have the following equivalences:
			\begin{enumerate}
				\item \begin{equation}
					f'^* g'_! \cong g_!f^*  
				\end{equation}
				\item
				\begin{equation}
					f'_*g'^{!} \cong g^!f_*
				\end{equation}
			\end{enumerate}
			\item (Proper pushforward) If $f:\X \to \Y$ is a compactifiable morphism, then there exists a natural transformation:
			\begin{equation} 
				\alpha_f: f_! \to f_* 
			\end{equation}
			which is an equivalence if $f$ is proper.
			
			\item (Purity) For $f$ to be  representable, smooth,  and separated of finite type, there exists a self equivalence $\op{Tw}_f$ and an equivalence 
			\[ \op{Tw}_f \circ f^! \cong f^* \]
			\item (Localization) For $i: \U \to X$ to be an open immersion and $j: \Z:= \X -\U \to \X$ to be the closed immersion from the complement of $\U$, we have the cofiber sequences:
			\begin{enumerate}
				\item \begin{equation}
					j_!j^! \to \op{id} \to i_*i^* 
				\end{equation}
				\item \begin{equation}
					i_!i^! \to  \op{id} \to j_*j^*
				\end{equation}
			\end{enumerate}
			\item (Homotopy Invariance) Let $\pi: \A^1_{\X} \to \X$ be the projection map. Then $\pi^*$ is fully faithful.
			\item $\Shext$ satisfies descent with respect to Nisnevich-local sections in $\Nst$. 
			
		\end{enumerate}
	\end{theorem}
	\begin{proof}
		The theorem follows from \cref{thmmain}, \cref{thmmain2}, \cref{localizationext}, \cref{hominvext}, \cref{alphaext} and \cref{purityext}. 
	\end{proof}

	\bibliography{thesisarchive.bib}

\end{document}